\documentclass[abstract=true,10pt]{scrartcl}

\RequirePackage{amsthm,amsmath,amsfonts,amssymb}
\RequirePackage{graphicx}

\usepackage[left=2cm,right=2.5cm]{geometry}
\usepackage{authblk} 
\usepackage{epsfig}
\usepackage{color}
\usepackage{dsfont}
\usepackage{bm}
\usepackage{verbatim}
\usepackage{stmaryrd}
\usepackage{bbm}			
\usepackage{enumerate} 
\usepackage{nicefrac}
\usepackage{xfrac}
\usepackage{units} 
\usepackage{mathtools}
\usepackage{pdfcomment}
\usepackage{url}
\usepackage{soul}
\usepackage[labelfont={bf,sf},font={footnotesize},labelsep=space,format=plain]{caption}
\usepackage{acronym}
\usepackage{mathrsfs}
\usepackage[T1]{fontenc}
\usepackage{caption, booktabs}
\usepackage{textcomp}
\usepackage{multirow}
\usepackage{lipsum}
\usepackage{makecell}
\usepackage[section]{placeins}
\usepackage{hologo}
\usepackage{xcolor}
\usepackage{dutchcal}

 \theoremstyle{plain}
 \newtheorem{theorem}{Theorem}[section]
 \newtheorem{Lemma}[theorem]{Lemma}
 \newtheorem{Cor}[theorem]{Corollary}

 \theoremstyle{definition}

 \newtheorem{Rem}[theorem]{Remark}
 \newtheorem{?}[theorem]{Problem}
 \newtheorem{Ex}[theorem]{Example}

\providecommand{\keywords}[1]
{
  \small	
  \textbf{\textit{Keywords---}} #1
}

\makeatletter
\renewcommand{\maketitle}{\bgroup\setlength{\parindent}{0pt}
\begin{center}
  \textbf{\@title}
\end{center}

\begin{flushleft}
	\@author
\end{flushleft}\egroup
}
\makeatother

\title{\begin{LARGE}On differentiability and mass distributions of typical bivariate copulas\end{LARGE}}

\date{}

\author[1,a]{Nicolas Dietrich}
\author[1,b]{Wolfgang Trutschnig}
\affil[1]{\begin{small} Department of Artificial Intelligence and Human Interfaces, University of Salzburg, Austria \end{small}}

\affil[a]{\begin{small}\url{nicolaspascal.dietrich@plus.ac.at}\end{small}}
\affil[b]{\begin{small}\url{wolfgang.trutschnig@plus.ac.at}\end{small}}

\begin{document}
\maketitle

\begin{abstract}
Despite the fact that copulas are commonly considered as analytically smooth/regular objects, derivatives of 
copulas have to be handled with care.  
Triggered by a recently published result characterizing multivariate copulas via $(d-1)$-increasingness of their 
partial derivative we study the bivariate setting in detail and show that the set of non-differentiability points 
of a copula may be quite large. 
We first construct examples of copulas $C$ whose first partial derivative $\partial_1C(x,y)$ 
is pathological in the sense that for almost every $x \in (0,1)$ it does not exist on a dense subset of $y \in (0,1)$, 
and then show that the family of these copulas is dense. 
Since in commonly considered subfamilies more regularity might be typical, we then focus on bivariate Extreme Value 
copulas (EVC) and show that a topologically typical EVC is not absolutely continuous but has
degenerated discrete component, implying that in this class typically $\partial_1C(x,y)$ exists in full $(0,1)^2$.
Considering that regularity of copulas is closely related to their mass distributions we 
then study mass distributions of topologically typical copulas and prove the 
surprising fact that topologically typical bivariate copulas are mutually completely dependent with full support.
Furthermore, we use the characterization of EVCs in terms of their associated Pickands dependence measures 
$\vartheta$ on $[0,1]$, show that regularity of $\vartheta$ carries over to the corresponding EVC and prove 
that the subfamily of all EVCs whose absolutely continuous, discrete and singular component has full support 
is dense in the class of all EVCs.   
\end{abstract}


\section{Introduction}
Constituting the link between multivariate distribution functions and their univariate marginals
 (see \cite{sklar}) as well as their resulting prominent role in the context of modeling stochastic dependence (see \cite{dur_princ,jaw2013,jaw2010,joe2014,Mai2012,mul2011,nelsen2006,rueschen2013}), over the past decades copulas 
have become an essential tool in probability theory and statistics, both from a theoretical as well 
as from an applied perspective. Considering that bivariate copulas are Lipschitz continuous, Rademacher's 
Theorem guarantees the
existence of their partial derivatives almost everywhere in the sense of the two-dimensional Lebesgue-measure $\lambda_2$. 
The exception set in Rademacher's Theorem applied to copulas can, however, be large in the sense of even having full 
support, a property that sometimes seems to have been overlooked or not handled with sufficient care. \\  
The afore-mentioned exception set is also relevant in the context of a recently published paper 
(see \cite[Corollary 4.2]{diff_char}) characterizing copulas in terms of $d$-monotonicity of their partial derivative. 
It is straightforward to show that in general there exists no set $D \subseteq \mathbb{I}$ with $\lambda(D)=1$ 
($\lambda$ denoting the one-dimensional Lebesgue measure) such that for every 
$x \in D$ the partial derivative $\partial_1C(x,y)$ of a bivariate copula $C$ exists for all $y \in D$.
We therefore first show the existence of bivariate copulas $C$ fulfilling that for $\lambda$-almost 
every $x \in \mathbb{I}$ there exists some 
$y_x \in \mathbb{I}$ such that the partial derivative $\partial_1C(x,y_x)$ does not exist and then 
construct more pathological examples exhibiting the property that for $\lambda$-almost every $x \in \mathbb{I}$ 
there exists a countably dense set $\mathcal{Q}_x \subseteq \mathbb{I}$ such that $\partial_1C(x,y)$ 
does not exist for any $y \in \mathcal{Q}_x$. Illustrating the fact that such pathological behavior 
might not be as uncommon as expected we then prove that the family of all such copulas is dense 
in the full class $\mathcal{C}$ of all bivariate copulas equipped with the standard uniform metric $d_\infty$. \\
Turing towards mixed partial derivatives $\partial_1\partial_2C$ and $\partial_2\partial_1C$ of a copula $C$ and recalling 
\cite[Theorem 2.2.8.]{nelsen2006} it is a direct consequence of Schwarz's theorem from calculus that 
in the case that $\partial_1\partial_2C$ and $\partial_2\partial_1C$ exist and are continuous everywhere on $(0,1)^2$, 
they coincide. Since in the general setting the mixed partial derivatives do not even exist everywhere, the interplay
between $\partial_1\partial_2C$ and $\partial_2\partial_1C$ is less obvious in general. Working with 
Markov kernels (regular conditional distributions) we clarify this interrelation and present a generalized 
version of Schwarz's theorem for copulas extending \cite[Theorem 2.2.8.]{nelsen2006}.

Viewing the family of copulas from a purely topological perspective and considering that the above-mentioned 
pathological examples are completely dependent (or convex combinations of completely dependent) copulas the question 
naturally arises, whether the family of completely dependent copulas is `small' or `large' in the sense of being atypical 
or typical. Working with Baire categories (see e.g. \cite{oxtoby1980}), topology offers a natural way to differentiate 
between `small' and `large' sets. We call a given subset of a topological space $(\mathcal{T},\tau)$ \textit{nowhere dense} if, and only if the interior of its closure is empty. A subset of $(\mathcal{T},\tau)$ is referred to as \textit{meager}/of 
\textit{first (Baire) category}, if it can be covered by a countable union of nowhere dense sets. A set is of \textit{second (Baire) category} if it is not meager. Moreover, a set is called \textit{co-meager} if it is the complement of a meager set.
Proceeding as in \cite{bruck1997} and returning to the concept of `small' and `large' sets, in a complete metric space meager
sets are the `small' sets, sets of second category are interpreted as `not small' and co-meager sets are the `large' sets. 
Following this very interpretation, in what follows we will refer to elements of a co-meager set as \textit{typical} 
and to elements of a meager set as \textit{atypical}. \\
Working with the afore-mentioned topological concepts C.W. Kim (see \cite{kim}) proved the striking
result that a typical bivariate copula is (mutually) completely dependent (see the subsequent section for a definition).
Building upon his result, in Section \ref{section:der_copulas} we prove a slightly stronger result: A typical copula 
$C$ is mutually completely dependent and (its associated doubly stochastic measure $\mu_C$) has full support. 
For a variety of manuscripts studying copulas in the context of Baire categories we refer to \cite{durante2022,cat_exchange,cat_quas_cop,typ_cop_sing} and the references therein.

Returning to the existence of $\partial_1C(x,y)$ one might naturally conjecture that in commonly considered 
subclasses more regularity might be typical. Here we therefore consider the well-known family of bivariate 
Extreme-Value copulas $\mathcal{C}_{ev}$ (EVC, for short) and study regularity in this class. 
Due to their simple algebraic form and practical aspects EVCs are particularly applied in 
finance and hydrology (see, e.g.,  \cite{McNeil2005,nelsen2006,salva2007}). 
Tackling regularity results and mass distributions of EVCs we first study the projection of the spectral 
measure $H$, defined on the unit simplex (see \cite{beirlant2004, falk2011, haan1977, twan}), to $\mathbb{I}$, refer to 
this measure $\vartheta$ as Pickands dependence measure, and then show how singularity/regularity properties of 
$\vartheta$ carry over to the corresponding EVC. Doing so, now allows to derive some of 
the results given in \cite{Mai2011} and \cite{evc-mass} in a simplified manner but also opens the door 
to proving the fact that EVCs whose discrete, singular and absolutely continuous component 
have full support are dense in $(\mathcal{C}_{ev},d_\infty)$. \\
Finally, again working with Pickands dependence measures we prove that - contrasting the fact that typical copulas in $(\mathcal{C},d_\infty)$ are mutually completely dependent and hence discrete in the sense introduced in the next
section - typical EVCs have degenerated discrete component, are not absolutely continuous, but have full support. 
In particular, for typical EVCs  $\partial_1C(x,y)$ exists in full $(0,1)^2$; EVCs can therefore be 
considered as typically quite regular. \\

The remainder of this paper is organized as follows: Section 2 contains notation and preliminaries 
that are used throughout the text. Section 3 studies differentiability of bivariate copulas and shows that the 
family of all bivariate copulas exhibiting the property that for $\lambda$-almost every $x \in (0,1)$
there exists a countable dense set $\mathcal{Q}_x \subseteq \mathbb{I}$ such that 
$\partial_1C(x,y)$ does not exist for any $y \in \mathcal{Q}_x$, is dense in $(\mathcal{C},d_\infty)$.  
The afore-mentioned extension of Schwarz's theorem for bivariate copulas is proved and the surprising and counter-intuitive fact, saying that a topologically typical copula $C$ in $(\mathcal{C},d_\infty)$ 
is mutually completely dependent and has full support, is established. 
Section 4 focuses on Extreme Value copulas, first recalls and then slightly extends the one-to-one-to-one interrelation 
between EVCs, Pickands dependence functions and Pickands measures. It is proved that various 
properties of the Pickands dependence measure $\vartheta$ carry over to the corresponding EVC and that the family 
of all EVCs whose discrete, singular and absolutely continuous component have full support are dense in 
$(\mathcal{C}_{ev},d_\infty)$.  
Establishing the result that typical EVCs have degenerated discrete component, are not absolutely continuous, 
but have full support, concludes Section 4. The latter results also close the circle since they imply that 
for typical EVCs $C$ the partial derivative $\partial_1C(x,y)$ exists in full $(0,1)^2$. \\
In order to simplify reading some technical lemmas and proofs have been shifted to the Appendix. 
Several additional examples and graphics illustrate the chosen procedures and some underlying ideas.

\clearpage

\section{Notation and Preliminaries}
In the sequel $\mathcal{C}$ denotes the family of all bivariate copulas. Given a copula $C\in\mathcal{C}$, as usual the corresponding doubly stochastic measure will be denoted by $\mu_C$, i.e., $\mu_C([0,x] \times [0,y]) := C(x,y)$ 
for all $(x,y) \in \mathbb{I}^2$ with $\mathbb{I} := [0,1]$. A copula $C \in \mathcal{C}$ is called exchangeable if 
the transposed copula $C^t$ coincides with $C$, i.e., if $C^t(x,y) := C(y,x)=C(x,y)$ holds for all $x,y \in \mathbb{I}$. $\mathcal{C}_e$ denotes the family of all exchangeable copulas. \\
For an arbitrary topological space $(S,\tau)$ we denote the Borel $\sigma$-field on $S$ by $\mathcal{B}(S)$, and 
let $\mathcal{P}(S)$ denote the family of all probability measures on $\mathcal{B}(S)$. Considering weak convergence of probability measures, we denote the topology induced by the afore-mentioned notion of convergence by $\tau_w$.
Moreover, for an arbitrary measure $\nu$ on $\mathcal{B}(S)$ the support of $\nu$, i.e., the complement of the union of all open sets $U$ with the property that $\nu(U) = 0$, will be denoted by $\mathrm{supp}(\nu)$.
Considering a set $E \subseteq S$, we write $\overline{E}$ for the topological closure of the set $E$ and 
$\mathrm{int}(E)$ for the interior of $E$.
Throughout this contribution, the support of a copula $C \in \mathcal{C}$ by definition will be the support of its corresponding doubly stochastic measure $\mu_C$. Considering the uniform metric $d_\infty$ on $\mathcal{C}$ it is well-known that $(\mathcal{C}, d_\infty)$ is a compact metric space.
For more background on copulas and doubly stochastic measures we refer to \cite{dur_princ,nelsen2006}.

The Lebesgue-measure on $\mathcal{B}(\mathbb{I}^2)$ is denoted by $\lambda_2$, for the univariate Lebesgue measure
we will simply write $\lambda$. For every $x \in S$ we denote the Dirac measure in $x \in S$ by $\delta_x$. 
Furthermore, the space of all Lebesgue integrable functions on $\mathbb{I}^d$ will be denoted by $L^1(\mathbb{I}^d,\mathcal{B}(\mathbb{I}^d),\lambda_d)$ for every $d \in \mathbb{N}$.
Considering two metric spaces $(S,d)$ and $(S',d')$, a Borel-measurable transformation \mbox{$T:S \rightarrow S'$} and a probability measure $\nu \in \mathcal{P}(S)$, the push-forward (measure) $\nu^T$ of $\nu$ via $T$ is defined by
$\nu^T(F):=\nu(T^{-1}(F))$ for all $F \in \mathcal{B}(S')$.

Throughout this contribution conditional distributions and Markov kernels will play a prominent role. 
A \emph{Markov kernel} from $\mathbb{R}$ to $\mathbb{R}$ is a map $K: \mathbb{R}\times\mathcal{B}(\mathbb{R}) \rightarrow \mathbb{I}$ fulfilling
that (i) for every fixed $E\in\mathcal{B}(\mathbb{R})$ the function $x\mapsto K(x,E)$ is Borel-measurable and 
\mbox{(ii) for} every $x\in\mathbb{R}$ the map $E\mapsto K(x,E)$ is a probability measure on $\mathcal{B}(\mathbb{R})$. 
If in (ii) we only have that the measure $E\mapsto K(x,E)$ fulfills $K(x,\mathbb{I})\leq 1$ (instead of $K(x,\mathbb{I})=1$),
then $K$ is called sub-Markov kernel. \\
Suppose that $(X,Y)$ is a random vector on a probability space $(\Omega,\mathcal{A},\mathbb{P})$. 
Then a Markov kernel $K(\cdot,\cdot)$ will be called a regular conditional distribution of
$Y$ given $X$ if for every set $E \in \mathcal{B}(\mathbb{R})$ the equation
$$
K(X(\omega), E) = \mathbb{E}(\mathbf{1}_E \circ Y | X)(\omega)
$$
holds for $\mathbb{P}$-almost every $\omega \in \Omega$.
It is a well-established fact that for each pair $(X,Y)$ of random variables, a regular conditional 
distribution $K(\cdot, \cdot)$ of $Y$ given $X$ exists and is unique for $\mathbb{P}^{X}$-almost every $x\in \mathbb{R}$.
Assuming that $(X,Y)$ has distribution function $C \in \mathcal{C}$ (more precisely, the natural extension of 
$C$ to $\mathbb{R}^2$ is the distribution function of $(X,Y)$), we will write $(X, Y)\sim C$, let
 $K_{C}:\mathbb{I} \times \mathcal{B}(\mathbb{I}) \to \mathbb{I}$ denote (a version of) the regular conditional 
 distribution of $Y$ given $X$ and call it the \emph{Markov kernel} of $C$.
Fixing $x \in \mathbb{I}$ and defining the $x$-section $G_x$ of an arbitrary set $G \in\mathcal{B}(\mathbb{I}^2)$ by $G_{x}:=\{y
\in \mathbb{I}: (x,y) \in G\}\in\mathcal{B}(\mathbb{I})$ applying \emph{disintegration} (see \cite[Section 5]{Kallenberg} and \cite[Section 8]{Klenke}) yields that
\begin{align}\label{eq:DI}
	\mu_C(G) = \int_{\mathbb{I}} K_{C}(x,G_x)
	\, \mathrm{d}\lambda(x).
\end{align}

In what follows, a measure $\nu$ on $(\mathbb{I}^d,\mathcal{B}(\mathbb{I}^d))$ with $d \in \mathbb{N}$ 
will be called singular (w.r.t $\lambda_d$) if, and only if (i) $\nu$ has no point masses and (ii) 
there exists some set $G \in \mathcal{B}(\mathbb{I}^d)$ with $\lambda_d(G) = 0$, $\nu(G) = \nu(\mathbb{I}^d)$. 
Obviously the doubly stochastic measure $\mu_C$ associated with a copula $C \in \mathcal{C}$ always has degenerated discrete component (in the sense that $\mu_C$ has no point masses). Following \cite{mult_arch}, however,
and using the Lebesgue-decomposition of the Markov kernel $K_C(x,\cdot)$ of $C$ into absolutely continuous, discrete and singular sub-kernels $K_C^{abs}(\cdot,\cdot), K_C^{dis}(\cdot,\cdot), K_C^{sing}(\cdot,\cdot): \mathbb{I} \times \mathcal{B}(\mathbb{I})
\rightarrow \mathbb{I}$, respectively, i.e.,
\begin{equation}\label{eq:lebesgue_decomp_markov}
K_C(x,F) = K_C^{abs}(x,F) + K_C^{dis}(x,F) + K_C^{sing}(x,F)
\end{equation}
for $x \in \mathbb{I}$ and $F \in \mathcal{B}(\mathbb{I})$, see \cite{Lange}, allows for a very natural definition of the
 absolutely continuous, the discrete and the singular component of $C$. In fact, working with disintegration and 
 equation \eqref{eq:lebesgue_decomp_markov}, we can define the absolutely continuous, the discrete and the singular 
 components $\mu_C^{abs}$, $\mu_C^{dis}$ and $\mu_C^{sing}$ of $\mu_C$ by
\begin{align}\label{eq:def_abs_dis_sing_copula}
\nonumber\mu_C^{abs}(E\times F) &:= \int_E K_C^{abs}(\mathbf{x},F) \mathrm{d}\lambda(x),\\
\nonumber\mu_C^{dis}(E\times F) &:= \int_E K_C^{dis}(\mathbf{x},F) \mathrm{d}\lambda(x),\\
\mu_C^{sing}(E\times F) &:= \int_E K_C^{sing}(\mathbf{x},F) \mathrm{d}\lambda(x),
\end{align}
for every $E \in \mathcal{B}(\mathbb{I})$ and every $F \in  \mathcal{B}(\mathbb{I})$, and extend them to full
$\mathcal{B}(\mathbb{I})$ in the standard way. 
Throughout this paper we call a copula $C$ absolutely continuous, discrete or singular if, and only if $\mu_C^{abs}(\mathbb{I}^2)=1$, $\mu_C^{dis}(\mathbb{I}^2)=1$ or $\mu_C^{sing}(\mathbb{I}^2)=1$, respectively. Moreover,  
$\mu_C^{abs}$, $\mu_C^{dis}$, $\mu_C^{sing}$ will be referred to as the absolutely continuous, the discrete, and the 
singular components of $C$ (or $\mu_C$), respectively.\\
Working particularly in the context of quantifying the extent of dependence of a random variable $Y$ on a 
random variable $X$, stronger metrics than $d_\infty$ have to be considered in $\mathcal{C}$: The Markov-kernel 
based metrics $D_p$, first introduced in \cite{sechser_paper}, are given by
$$
D_p(A,B) := \left(\int_{\mathbb{I}^2}|K_A(x,[0,y]) - K_B(x,[0,y])|^p \mathrm{d}\lambda_2(x,y)\right)^\frac{1}{p},
$$
for $p \in [1,\infty)$ and
$$
D_\infty(A,B) := \sup_{y \in \mathbb{I}}\int_{\mathbb{I}}|K_A(x,[0,y]) - K_B(x,[0,y])|\mathrm{d}\lambda(x),
$$
for $p=\infty$ and arbitrary $A,B \in \mathcal{C}$. As shown in \cite{sechser_paper} the topologies induced by the 
metrics $D_p$ are all equivalent and the the topology induced by $D_p$ (for an arbitrary $p \in [0,\infty]$) 
is strictly finer than the one induced by $d_\infty$. In other words: For fixed $p \in [1,\infty]$ and copulas
$C,C_1,C_2, ...\in \mathcal{C}$ we have that $D_p(C_n,C) \overset{n \rightarrow \infty}{\longrightarrow} 0$
implies $d_\infty(C_n,C) \overset{n \rightarrow \infty}{\longrightarrow} 0$ but not necessarily vice versa, see \cite{sechser_paper} for a counter-example. \\
From a statistical perspective, one might even consider a stronger but natural notion of convergence, 
that of weak convergence of Markov kernels (wcc, for short), first introduced in \cite{bernoulli}: 
Suppose that $C,C_1,C_2,\ldots$ are copulas with associated Markov kernels $K_C,K_{C_1},K_{C_2},...$. 
We say that the sequence $(C_n)_{n \in \mathbb{N}}$ converges weakly conditional to $C$ if, and only if for $\lambda$-almost every $x \in \mathbb{I}$ the sequence $(K_{C_n}(x,\cdot))_{n \in \mathbb{N}}$ of probability measures on $\mathcal{B}(\mathbb{I})$ converges weakly to the probability measure $K_{C}(x,\cdot)$. In the following we write $C_n \overset{wcc}{\longrightarrow} C$ to denote weak conditional convergence of $(C_n)_{n \in \mathbb{N}}$ to $C$. 
It is straightforward to show that wcc implies $D_p$ convergence but not necessarily vice versa 
(again see \cite{bernoulli}). \\
We call a map $h \colon \mathbb{I} \rightarrow \mathbb{I}$ $\lambda$-preserving if, and only if $\lambda^h(F) = \lambda(F)$ holds for every $F \in \mathcal{B}(\mathbb{I})$. 
A copula $C \in \mathcal{C}$ is called completely dependent if there exists some $\lambda$-preserving transformation $h$ such that $C$ concentrates its mass on the graph of $h$, or, equivalently, if its corresponding Markov kernel is given by $K_C(x,F) = \mathbf{1}_F(h(x))$ for $\lambda$-almost all $x \in \mathbb{I}$ (see \cite{sechser_paper} for more 
equivalent formulations). 
If, in addition, $h$ is bijective, we call the associated completely dependent copula mutually completely dependent. Throughout this contribution we will denote the family of all completely dependent copulas by $\mathcal{C}_d$ 
and the family of all mutually completely dependent copulas by $\mathcal{C}_{mcd}$. 
A $\lambda$-preserving map $h \colon \mathbb{I} \rightarrow \mathbb{I}$ will be called a (classical) equidistant even shuffle of $M$ with $N \in \mathbb{N}$ stripes if, and only if $h$ is linear with slope $1$ on each interval $I_N^i := (\frac{i-1}{N},\frac{i}{N})$, injective on $\bigcup_{i=1}^N(\frac{i-1}{N},\frac{i}{N})$ and just permutes the intervals $I_N^1,...,I_N^N$. In the sequel $\mathcal{S}_N$ will denote the family of all completely dependent copulas whose corresponding $\lambda$-preserving transformation is an even shuffle with $N$ stripes. The space of all completely dependent copulas 
whose associated $\lambda$-preserving transformation is an equidistant shuffle will be denoted by $\mathcal{S} := \bigcup_{N \in \mathbb{N}} \mathcal{S}_N$. For more information on completely dependent copulas we refer to 
\cite{nelsen2006,sechser_paper}.\\
The last concept we will need is that of a checkerboard copula: Suppose that $B \in \mathcal{C}$ and $N \in \mathbb{N}$
are arbitrary but fixed and that $T=(t_{i,j}^N)_{i,j=1}^N$ is a matrix fulfilling that $(N \cdot t_{i,j}^N)_{i,j=1}^N$ is doubly stochastic. Furthermore, for every $(i,j) \in \{1,\ldots,N\}^2 $ define the affine 
transformation $w_{i,j}^N: \mathbb{I}^2 \rightarrow R_{i,j}:=I_i^N \times I_j^N$ by 
$$
w_{i,j}^N(x,y) = \left(\frac{i-1}{N} + \frac{x}{N},\frac{j-1}{N}+\frac{y}{N}\right).
$$
Then the copula $CB_T^{B}$, defined implicitly via 
\begin{equation}\label{eq:checkerboard}
\mu_{CB_T^{B}}(G) := \sum_{i,j=1}^Nt_{i,j}^{N}\mu_B^{w_{i,j}^N}(G \cap R_{ij}), \quad G \in \mathcal{B}(\mathbb{I}^2)
\end{equation}
is called the $T$-checkerboard of $B$. To simplify notation we will let $\mathcal{CB}_{N}^{B}$ denote the 
family of all $CB_T^{B}$ with $T$ being an $N \times N$ matrix as described above and write 
$$\mathcal{CB}^B := \bigcup_{N \in \mathbb{N}}\mathcal{CB}_{N}^{B}.$$
For a given copula $A \in \mathcal{C}$ the $N$-checkerboard-$B$ approximation of $A$ is defined via 
\begin{equation}\label{eq:checkerboard_approx}
\mu_{CB_N^{B}(A)}(G) = \sum_{i,j=1}^N\mu_A(R_{i,j}^{N})\mu_B^{w_{i,j}^{N}}(G \cap R_{ij}^{N})
\end{equation}
for arbitrary $G \in \mathcal{B}(\mathbb{I}^2)$. For more background on checkerboards we refer to \cite{bernoulli}
and the references therein.
\clearpage
\section{Differentiability of copulas and typical mass distributions} \label{section:der_copulas}
We first derive various results on differentiability and then study regularity of topologically typical copulas.
\subsection{Differentiability results}
Corollary 4.2 in \cite{diff_char} translated to the bivariate setting 
states that a function $C \colon \mathbb{I}^2 \rightarrow \mathbb{I}$ that is grounded and has uniform marginals 
and absolutely continuous sections is a bivariate copula if, and only if its first partial derivative 
$y \mapsto \partial_1C(x,y)$ is non-decreasing for $\lambda$-almost all $x \in \mathbb{I}$.
Here we focus on differentiability and start with providing examples of copulas $C \in \mathcal{C}$ fulfilling 
the property that for $\lambda$-almost every $x \in \mathbb{I}$ there exists some point $y = y_x \in \mathbb{I}$ such that $\partial_1C(x,y)$ does not exist. Throughout this contribution we denote the family of all such copulas by $\mathcal{C}_p$.
The subsequent example shows that all equidistant shuffels are elements of $\mathcal{C}_p$.
\begin{Ex}[Non-differentiability in one point] \label{Ex:Shuffels}
Suppose that $C \in \mathcal{S}$. Then there exists some $N \in \mathbb{N}$ with $C \in \mathcal{S}_N$. 
Let $x \in \mathbb{I} \setminus \bigcup_{i = 0}^N\{\frac{i}{N}\}$ be arbitrary but fixed. Then 
there exists some $j \in \{1,...,N\}$ with $x \in (\frac{j-1}{N},\frac{j}{N})$. Letting $h$ denote the corresponding 
$\lambda$-preserving transformation, setting $y = h(x)$, then using the fact that $h$ is monotonically increasing on $(\frac{j-1}{N},\frac{j}{N})$ and calculating the right- and left-hand partial derivative, respectively, yields
\begin{align*}
\partial_1^{+}C(x,h(x)) &= \lim_{s \downarrow 0}\frac{C(x + s,h(x)) - C(x,h(x))}{s} = 0
\end{align*}
as well as
$$
\partial_1^{-}C(x,h(x)) \lim_{s \uparrow 0}\frac{C(x + s,h(x)) - C(x,h(x))}{s} = 1.
$$
The left-hand and right-hand derivatives exist but do not coincide, implying that $\partial_1 C(x,T(x))$ does not exist by definition. Figure \ref{fig:shuffles} depicts two examples of elements of shuffles and points of non-differentiability.  
\end{Ex}
\begin{figure}[!ht]
	\centering
	\includegraphics[width=0.9\textwidth]{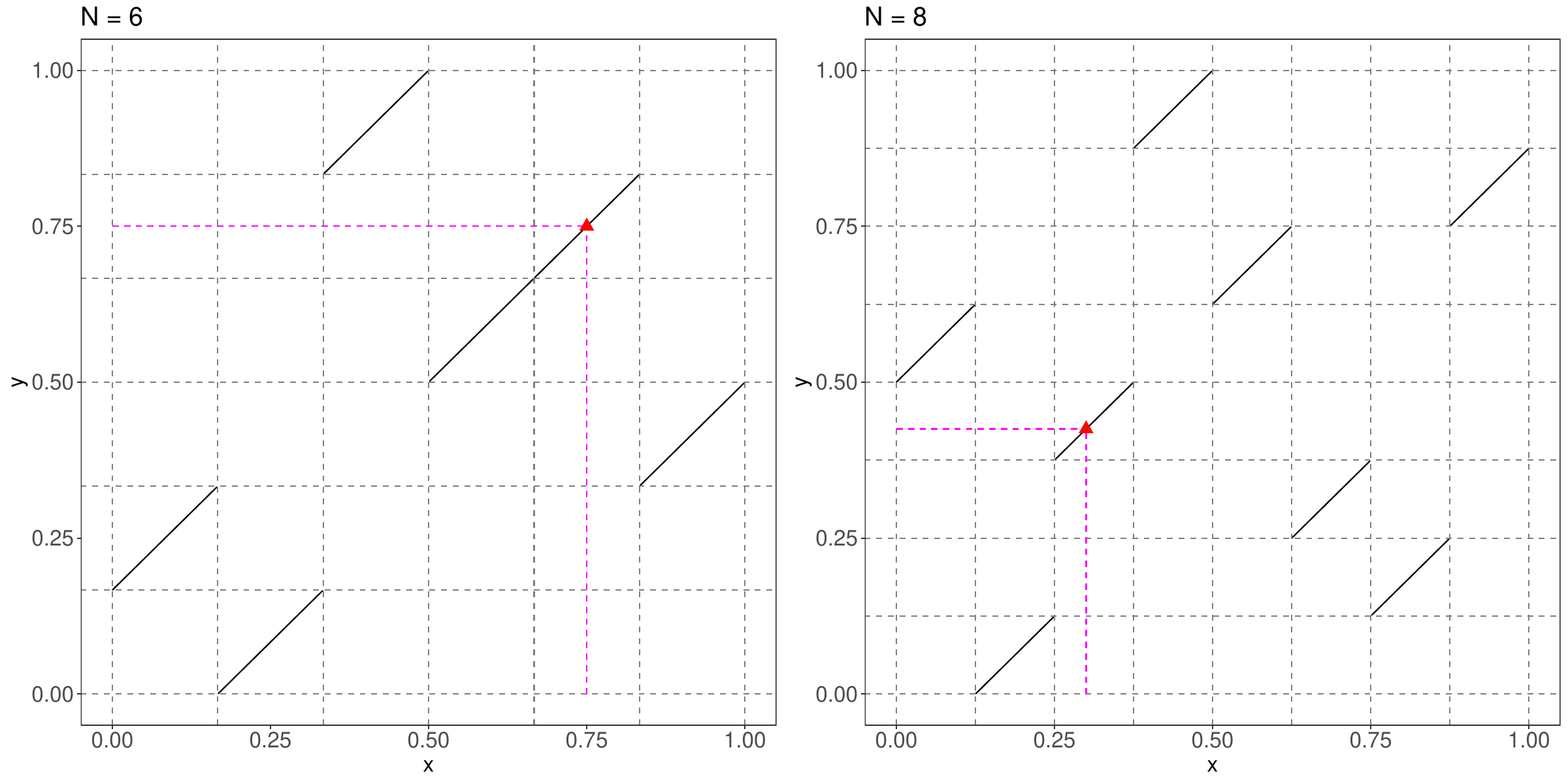}
	\caption{Two examples of equidistant shuffles as studied in Example \ref{Ex:Shuffels}. The red triangles depict 
	points where the partial derivative w.r.t. $x$ does not exist.}
    \label{fig:shuffles}
    \end{figure}
The family of points of non-differentiability may be much larger - much more pathological cases than the one mentioned in Example \ref{Ex:Shuffels} exist: As a second example we show the existence of a copula $C$ with the following property: 
There exists some set $\Lambda \in \mathcal{B}(\mathbb{I})$ with $\lambda(\Lambda) = 1$ such that for every $x \in \Lambda$ there exists a dense set $\mathcal{Q}_x \subset \mathbb{I}$ such that $\partial_1C(x,y)$ does not exist for any $y\in \mathcal{Q}_x$. In the sequel the family of all copulas having this property will be denoted by $\mathcal{C}_{\mathcal{Q}}$.
\begin{Ex}[Non-differentiability on a dense subset] \label{ex:rotation}
Consider the $\lambda$-preserving rotations $R_r(x) = x + r \text{ (mod 1)}$ for $r \in [0,1)$ and define 
$K \colon \mathbb{I} \times \mathcal{B}(\mathbb{I}) \rightarrow \mathbb{I}$ by
$$
K(x,E) := \sum_{n \in \mathbb{N}}\frac{1}{2^n}\delta_{R_{r_n}}(E) =  \sum_{n \in \mathbb{N}}\frac{1}{2^n}\mathbf{1}_E(R_{r_n}(x))
$$
for $\{r_1,r_2,...\} = \mathbb{Q}_{0,1} := \mathbb{Q} \cap [0,1)$ where $r_i \neq r_j$ for all $i \neq j$. Then
$K$ is obviously measurable in the first and a probability measure in the second argument, i.e., $K(\cdot,\cdot)$ is a Markov kernel from $\mathbb{I}$ to $\mathbb{I}$. Applying monotone convergence together with the fact that the rotations $R_{r_n}$ are $\lambda$-preserving yields that
\begin{align*}
\int_{\mathbb{I}}K(x,E) \mathrm{d}\lambda(x) = \int_{\mathbb{I}} \sum_{n \in \mathbb{N}}\frac{1}{2^n}\mathbf{1}_E(R_{r_n}(x)) \mathrm{d}\lambda(x) = \sum_{n \in \mathbb{N}}\int_{\mathbb{I}}\frac{1}{2^n}\mathbf{1}_E(R_{r_n}(x)) \mathrm{d}\lambda(x) = \lambda(E),
\end{align*}
so $K(\cdot,\cdot)$ is the Markov kernel of a unique copula $C$. 
Fix  $x_0 \in \mathbb{I}\setminus \mathbb{Q}$ and set $y_j := R_{r_j}(x_0)$. Then obviously the set $\{y_j \colon j \in \mathbb{N}\}$ is dense in $\mathbb{I}$. Writing
\begin{equation}\label{eq:decomp_kernel}
K(x,[0,y_j]) = \underbrace{\sum_{n \neq j}\frac{1}{2^n}\mathbf{1}_{[0,y_j]}(R_{r_n}(x))}_{:=L(x,y_j)} + \underbrace{\frac{1}{2^j}\mathbf{1}_{[0,y_j]}(R_{r_j}(x))}_{:=G_{r_j}(x,y_j)}
\end{equation}
and using the fact that $x \mapsto \mathbf{1}_{[0,y_j]}(R_{r_n}(x))$ is continuous for every $n \neq j$, absolute convergence of the series $L(x,y_j)$ implies that $x \mapsto L(x,y_j)$ is continuous in $x_0$. Applying disintegration and using equation \eqref{eq:decomp_kernel} yields
\begin{equation*}
C(x,y_j) = \int_{[0,x]} K(s,[0,y_j]) \mathrm{d}\lambda(s) = \underbrace{\int_{[0,x]}L(s,y_j) \mathrm{d}\lambda(s)}_{=:A(x,y_j)} + \underbrace{\int_{[0,x]}G_{r_j}(s,y_j) \mathrm{d}\lambda(s)}_{=:B(x,y_j)},
\end{equation*}
for every $x \in \mathbb{I}$.
Calculating the right-hand partial derivative of $A$ in $x_0$, applying continuity of $x \mapsto L(x,y_j)$ in $x_0$ and using the fact that every point of continuity is a Lebesgue point (see \cite{rudin1974}) yields
$$
\partial_1^+A(x_0,y_j)= \lim_{h \downarrow 0}\frac{1}{h}\int_{[x_0,x_0+h]}L(s,y_j) \mathrm{d}\lambda(s) = L(x_0,y_j).
$$
Proceeding analogously yields that the left-hand derivative of $A$ in $x_0$ exists. 
Using the fact that $x_0 \neq 1-q_j$ and calculating the right-hand partial derivative of $B$ in $x_0$ yields 
$$
\partial_1^+ B(x_0,y_j) = \lim_{h \downarrow 0}\frac{1}{h}\int_{[x_0,x_0+h]}\frac{1}{2^j}\mathbf{1}_{[0,R_{r_j}(x_0)]}(R_{r_j}(s)) \mathrm{d}\lambda(s) = 0,
$$
and proceeding analogously for the left-hand partial derivative in $x_0$ we get
\begin{align*}
\partial_1^- B(x_0,y_j) =  \lim_{h \downarrow 0}\frac{1}{h}\int_{[x_0-h,x_0]}\frac{1}{2^j}\mathbf{1}_{[0,R_{r_j}(x_0)]}(R_{r_j}(s))  \mathrm{d}\lambda(s) = \frac{1}{2^j}.
\end{align*}
Altogether it follows that $\partial_1 C(x_0,y_j)$ does not exist. Considering that $x_0 \in \mathbb{I}\setminus\mathbb{Q}$ and $j \in \mathbb{N}$ were arbitrary completes the proof.
\end{Ex}
Working with the previous example in combination with checkerboard copulas allows to show that
 $\mathcal{C}_{\mathcal{Q}}$ is dense in $\mathcal{C}$, implying that pathological cases of non-differentiability 
 in the sense of Example \ref{ex:rotation} can be found `everywhere' in $\mathcal{C}$ even with respect to the finer
 topology induced by the metric $D_p$.
\begin{theorem}\label{thm:pathological_dense}
The set $\mathcal{C}_{\mathcal{Q}}$ is dense in $(\mathcal{C},D_p)$ for every $p \in [1,\infty]$.
\end{theorem}
\begin{proof}
Fix $C \in \mathcal{C}$ and $n \in \mathbb{N}$. Choose an arbitrary $B \in \mathcal{C}_\mathcal{Q}$ and consider the $B$-checkerboard approximation  $CB_N^B(C_n)$ of the copula 
 $C_n:=(1-\frac{1}{n})C + \frac{1}{n} \, \Pi$ according to equation \eqref{eq:checkerboard_approx}. Obviously 
$C_n$ has full support so it follows immediately that $CB_N^B(C_n) \in \mathcal{C}_{\mathcal{Q}}$ holds for every 
$N \in \mathbb{N}$. Furthermore, applying \cite[Theorem 3.2]{bernoulli} yields that $CB_N^B(C_n) \overset{wcc}{\longrightarrow} C_n$ as $N \rightarrow \infty$, implying $\lim_{n \rightarrow \infty} D_p(CB_N^B(C_n),C)=0$. 
Finally, considering that we also have weak conditional convergence
of $(C_n)_{n \in \mathbb{N}}$ to $C$ for $n \rightarrow \infty$ the result follows.
\end{proof}
Using the interrelations between $D_p$-convergence and convergence w.r.t $d_\infty$
Theorem \ref{thm:pathological_dense} has the following direct consequence:
\begin{Cor}\label{cor:pathological_dense}
The set $\mathcal{C}_{\mathcal{Q}}$ is dense in $(\mathcal{C},d_\infty)$.
\end{Cor}
The previous example(s) built upon discontinuity of the conditional distributions - the following result shows that without discontinuities of the conditional distribution functions non-existence of the partial derivative can not happen.
\begin{theorem}\label{thm:cont_ker}
    Suppose that $C \in \mathcal{C}$ fulfills that $\lambda$-almost all conditional distribution functions $y \mapsto K_C(x,[0,y])$ are continuous. Then there exists a set $\Lambda \in \mathcal{B}(\mathbb{I})$ with $\lambda(\Lambda) = 1$ such that
    \begin{equation}\label{eq:der_markov_kernel}
    \partial_1C(x,y) = K_C(x,[0,y])
    \end{equation}
    holds for every $x \in \Lambda$ and $y \in \mathbb{I}$.
\end{theorem}
\begin{proof}
By assumption there exists some
 $\Lambda_1 \in \mathcal{B}(\mathbb{I})$ with $\lambda(\Lambda_1) = 1$ such that $y \mapsto K_C(x,[0,y])$ is continuous for every $x \in \Lambda_1$. Fixing $y \in \mathbb{I}$, applying disintegration and Lebesgue's differentiation theorem 
 (see \cite{rudin1974}) yields the existence of some $\Lambda_y \in \mathcal{B}(\mathbb{I})$ with $\lambda(\Lambda_y) = 1$ such that
$$
\partial_1C(x,y) = K_C(x,[0,y])
$$ holds for all $x \in \Lambda_y$. Defining $\Lambda_2 := \bigcap_{q \in \mathbb{Q} \cap \mathbb{I}}\Lambda_q$
implies $\lambda(\Lambda_2) = 1$ as well as
\begin{equation}\label{eq:derv_equal_ker}
\partial_1C(x,q) = K_C(x,[0,q])
\end{equation}
for every $x \in \Lambda_2$ and every $q \in \mathbb{Q} \cap \mathbb{I}$.
Set $\Lambda := \Lambda_1 \cap \Lambda_2 \cap (0,1)$ and fix  $x \in \Lambda$. For every $y \in \mathbb{I}$ 
and $h>0$ sufficiently small define the difference quotient $I_h(x,y)$ by
$$
I_h(x,y) := \frac{C(x+h,y) - C(x,y)}{h}.
$$
Then using $2$-increasingness of $C$ directly yields
$$
I_h(x,q_1) \leq I_h(x,y) \leq I_h(x,q_2) 
$$
for arbitrary $q_1,q_2 \in \mathbb{Q} \cap \mathbb{I}$ with $q_1 \leq y \leq q_2$. According to equation \eqref{eq:derv_equal_ker} we have
$$
\lim_{h \rightarrow 0}I_h(x,q_1) = K_C(x,[0,q_1])
$$
as well as
$$
\lim_{h \rightarrow 0}I_h(x,q_2) = K_C(x,[0,q_2]).
$$
Fixing sequences $(q_n)_{n \in \mathbb{N}},(r_n)_{n \in \mathbb{N}}$ in $\mathbb{Q} \cap \mathbb{I}$ with 
$q_n\downarrow y$ and $r_n\uparrow y$ and $\varepsilon > 0$ arbitrarily, using continuity of $y \mapsto K_C(x,[0,y])$ there 
exists some $n_0 \in \mathbb{N}$ such that
$$
K_C(x,[r_n,q_n]) < \varepsilon
$$
for all $n \geq n_0$. Working with the upper Dini-derivative (see \cite{dur_princ}) yields
$$
K_C(x,[0,y]) - \varepsilon < K_C(x,[0,r_n]) \leq \limsup_{h \downarrow 0}I_h(x,y) \leq K_C(x,[0,q_n]) < K_C(x,[0,y]) + \varepsilon.
$$
Since $\varepsilon > 0$ was arbitrary $\limsup_{h \downarrow 0}I_h(x,y) = K_C(x,[0,y])$ follows.
Proceeding analogously for the lower Dini-derivative, we finally get that $\partial_1C(x,y) = K_C(x,[0,y])$ for every
 $y \in \mathbb{I}$. Since $x \in \Lambda$ was arbitrary this completes the proof.
\end{proof}
Reinterpreting Theorem \ref{thm:cont_ker} in the context of absolutely continuous, singular and discrete components of 
$C$, equation \eqref{eq:der_markov_kernel} obviously holds if $C$ has degenerated discrete component.
This implies the following result:
\begin{Cor}\label{cor:abs_cont}
        Suppose that $C \in \mathcal{C}$ fulfills $\mu_C^{dis}(\mathbb{I}^2) = 0$. Then there exists some set $\Lambda \in \mathcal{B}(\mathbb{I})$ with $\lambda(\Lambda) = 1$ such that
        $$
        \partial_1C(x,y) = K_C(x,[0,y])
        $$
        holds for all $x \in \Lambda$ and $y \in \mathbb{I}$.\\
        Moreover, if $C \in \mathcal{C}$ is absolutely continuous with density $f$ then there exists some set $\Lambda \in \mathcal{B}(\mathbb{I})$ with $\lambda(\Lambda) = 1$ such that for every $x \in \Lambda$ and every $y \in \mathbb{I}$ we have that
    $$
    \partial_1C(x,y) = \int_{[0,y]}f(x,t) \mathrm{d}\lambda(t).
    $$
\end{Cor}
\begin{proof}
(i) Assume that $\mu_C^{dis}(\mathbb{I}^2) = 0$. Then according to equation \eqref{eq:def_abs_dis_sing_copula}, using 
Lebesgue's decomposition Theorem for Markov kernels, we obtain the existence of a set $\Lambda \in \mathcal{B}(\mathbb{I})$ with $\lambda(\Lambda) = 1$ such that the identity $K_C(x,[0,y]) = K_C^{abs}(x,[0,y]) + K_C^{sing}(x,[0,y])$ holds for all $y \in \mathbb{I}$ and $x \in \Lambda$. Since both $y\mapsto K_C^{abs}(x,[0,y])$ and $y\mapsto K_C^{sing}(x,[0,y])$ are
 continuous for every $x \in \Lambda$, the desired result follows by applying Theorem \ref{thm:cont_ker}. 
The second assertion is a direct consequence of the fact that in the absolutely continuous case with density $f$ 
a version of the Markov kernel $K_C(\cdot,\cdot)$ of $C$ is give by $K_C(x,F)=\int_F f(x,t)d\lambda(t)$.
\end{proof}
If the mixed partial derivatives $\partial_1\partial_2C(x,y)$ and $\partial_2\partial_1C(x,y)$ exist and are continuous for all $x,y \in (0,1)$, then according to Schwarz's theorem they need to coincide, i.e.,
\begin{equation}\label{eq:change_derivative}
\partial_1\partial_2 C(x,y) = \partial_2\partial_1 C(x,y)
\end{equation}
for all $x,y \in (0,1)$ (see e.g. \cite[Theorem 2.2.8.]{nelsen2006}). Turning to the general setting, the mixed partial derivatives $\partial_1\partial_2C$ and $\partial_2\partial_1C$ do not need to exist everywhere in $(0,1)^2$ since, as pointed out in the previous examples, fixing $x$ in a set of full $\lambda$-measure, not even the partial 
derivative $\partial_1C(x,y)$ needs to exist for all $y \in \mathbb{I}$ and therefore equation \eqref{eq:change_derivative} cannot hold for arbitrary $C \in \mathcal{C}$. 
On the other hand, for arbitrary $C \in \mathcal{C}$ with Markov kernel $K_C$, working again 
with disintegration and applying Lebesgue's differentiation theorem yields that for fixed $y \in \mathbb{I}$ there exists a set $\Lambda_y \in \mathcal{B}(\mathbb{I})$ with $\lambda(\Lambda_y) = 1$ such that
$$
\partial_1C(x,y) = K_C(x,[0,y])
$$
holds for all $x \in \Lambda_y$. In other words: The partial derivative $\partial_1C(x,y)$ coincides with the Markov kernel $K_C(\cdot,\cdot)$ on a `good' set and thus $K_C(\cdot,\cdot)$ can be seen as a regularized version of $\partial_1C$. 
Building upon this fact, working with Markov kernels we are able to generalize the result in \cite[Theorem 2.2.8]{nelsen2006} to the whole family of bivariate copulas:
\begin{theorem}\label{lem:der_kernel}
    Suppose that $C \in \mathcal{C}$ is an arbitrary copula and let $K_C$ and $K_{C^t}$ denote Markov kernels of $C$ and its transpose $C^t$, respectively. Furthermore let $f \in L^1(\mathbb{I}^2,\mathcal{B}(\mathbb{I}^2),\lambda_2)$ denote the Radon-Nikodym derivative of $\mu_C$, i.e., the density of $\mu_C^{abs}$. Then the identity
    $$
    \partial_yK_C(x,[0,y]) = f(x,y) = \partial_xK_{C^t}(y,[0,x])
    $$
    holds for $\lambda_2$-almost every $(x,y) \in \mathbb{I}^2$.
\end{theorem}
\begin{proof}
Working with the decomposition in equation \eqref{eq:lebesgue_decomp_markov} and aggregating the discrete and the singular component in $K_C^\perp$ yields 
$K_C(x,F) = K_C^{\mathrm{abs}}(x,F) + K_C^{\perp}(x,F)$ as well as
    $
     K_{C^t}(x,F) = K_{C^t}^{\mathrm{abs}}(x,F) + K_{C^t}^{\perp}(x,F)
    $
    for $x \in \mathbb{I}$ and $F \in \mathcal{B}(\mathbb{I})$.
    Considering $K_C^\perp$, defining the set
    $$
    E := \{(x,y) \in \mathbb{I}^2\colon \partial_yK_C^{\perp}(x,[0,y]) = 0\},
    $$
 using the fact that $(x,y) \mapsto K_C^\perp(x,[0,y])$ is measurable function and that $y \mapsto K_C^\perp(x,[0,y])$ is non-decreasing, applying \cite[Theorem 1]{Moshe1977} yields that  
 $$\{(x,y) \in (0,1)^2\colon \partial_yK_C^{\perp}(x,[0,y]) \text{ exists}\}$$ is a Borel set.
    Hence, using singularity of $y \mapsto K_C^\perp(x,[0,y])$ implies $E \in \mathcal{B}(\mathbb{I}^2)$. 
    Fixing $x \in \mathbb{I}$ and working with $x$-cuts and Fubini's theorem yields $\lambda(E_x) = 1$ for $\lambda$-almost 
    every $x \in \mathbb{I}$ as well as
    $$
    \lambda_2(E) = \int_{\mathbb{I}}\lambda(E_x) \mathrm{d}\lambda(x) = 1.
    $$
    Proceeding analogously we obtain the existence of a set $\tilde{E} \in \mathcal{B}(\mathbb{I}^2)$ with $\lambda_2(\tilde{E}) = 1$ such that
    $
    \partial_xK_{C^t}^{\perp}(y,[0,x]) = 0
    $
    for all $(x,y) \in \tilde{E}$. Altogether there exists a set $Q := E \cap \tilde{E}$ with $\lambda(Q) = 1$ and $\partial_y K_C^\perp(x,[0,y]) = 0 = \partial_x K_{C^t}^\perp(y,[0,x])$ for all $(x,y) \in Q$. We will prove that
    $$
    \partial_y K_C^{\mathrm{abs}}(x,[0,y]) = f(x,y) = \partial_x K_{C^t}^{\mathrm{abs}}(y,[0,x])
    $$
    holds for $\lambda_2$-almost all $(x,y) \in \mathbb{I}^2$ with $f$ as in the theorem, and proceed as follows:
     Define a version of the absolutely continuous component $K_C^{abs}$ of $K_C$ by $K_C^{abs}(x,F) := \int_Ff(x,y) \mathrm{d}\lambda(y)$ for every $x \in \mathbb{I}$ and every $F \in \mathcal{B}(\mathbb{I})$. 
     Similarly to the case for $K_C^\perp$, defining
    $$
    G := \{(x,y) \in \mathbb{I}^2 \colon \partial_yK_C^{\mathrm{abs}}(x,[0,y]) \text{ exists and } \partial_yK_C(x,[0,y]) = f(x,y)\},
    $$
    and applying the same line of argumentation as done for showing measurability of $E$ yields Borel measurability of $G$. 
    Working with the fact that $y \mapsto K_C^{abs}(x,[0,y])$ is differentiable $\lambda$-almost everywhere, 
    applying Fubini's theorem yields $\lambda_2(G) = 1$.
    Considering the transposed copula $C^t$, using the identity $\mu_C^{abs}(F_1 \times F_2) = \mu_{C^t}^{abs}(F_2 \times F_1)$ for all $F_1,F_2 \in \mathcal{B}(\mathbb{I})$ and proceeding analogously to the previous case, there exists a set $\tilde{G} \in \mathcal{B}(\mathbb{I}^2)$ with $\lambda_2(\tilde{G}) = 1$ such that
    $$
    \partial_x K_{C^t}^{abs}(y,[0,x]) = f(x,y)
    $$
    holds for all $(x,y) \in \tilde{G}$. This implies that $\partial_y K_C^{\mathrm{abs}}(x,[0,y]) = f(x,y) = \partial_x K_{C^t}^{\mathrm{abs}}(y,[0,x])$ for all $(x,y) \in  M := G \cap \tilde{G}$. Using $\lambda(M) = 1$
    altogether yields $\lambda_2(Q \cap M) = 1$ as well as
    $$
    \partial_yK_C(x,[0,y]) = \partial_yK_C^{abs}(x,[0,y]) = f(x,y) = \partial_xK_C^{abs}(y,[0,x]) = \partial_xK_C(y,[0,x]),
    $$
    for all $(x,y) \in Q \cap M$. This completes the proof.
\end{proof}
\subsection{Baire category results for $\mathcal{C}$}
The copulas proposed in Example \ref{Ex:Shuffels} and Example \ref{ex:rotation} are completely dependent, 
complete dependence, however, may seem quite pathological. 
Nevertheless, using Baire categories it can be shown that complete dependence is much less pathological than
one might assume. In fact, C.W. Kim proved in \cite{kim} that topologically typical Markov-operators are induced by a 
$\lambda$-preserving map. Using the fact that the family of all Markov-operators 
(endowed with the weak operator topology) is isomorphic to $(\mathcal{C},d_\infty)$ (see \cite{darsow}), we obtain 
the following translation to $\mathcal{C}$: 
\begin{theorem}\label{thm:kim}
    $\mathcal{C}_{mcd}$ and $\mathcal{C}_{cd}$  are co-meager in $(\mathcal{C},d_\infty)$.
\end{theorem}
We conjecture that every $C \in \mathcal{C}_{mcd}$ is automatically an element of $\mathcal{C}_p$, i.e., that for 
$\lambda$-almost every $x \in \mathbb{I}$ there exists some $y \in \mathbb{I}$ such that $\partial_1 C(x,y)$ does not exist.
We have, however, not been able to prove or falsify this conjecture since $\lambda$-preserving bijections may exhibit 
quite irregular behavior going far beyond being piecewise linear: As shown in \cite{Feldman} it is possible to construct 
a $\lambda$-preserving transformation $h$ such that the induced completely dependent copula $C_d$ has full support.
Motivated by this result we now prove that typical bivariate copulas have full support.
\begin{Lemma}\label{lem:typical:cop_full_supp}
The family $\{C\in \mathcal{C}\colon\mathrm{supp}(\mu_C)=\mathbb{I}^2\}$ is co-meager in $(\mathcal{C},d_\infty)$.
\end{Lemma}
\begin{proof}
    We define the set $ \mathcal{A}_N$ by
    $$
    \mathcal{A}_N := \{C\in \mathcal{C} \colon \mathrm{supp}(\mu_C)^c \text{ contains a cube } Q \text{ with } \lambda_2(Q) \geq \tfrac{1}{N}\}
    $$
    and first show that $\mathcal{A}_N $ is closed w.r.t. $d_\infty$. Suppose that $(C_\ell)_{\ell \in \mathbb{N}}$ 
    is a sequence in $\mathcal{A}_N$ with $\lim_{\ell \rightarrow \infty} d_\infty(C_\ell,C) = 0$ for some 
    $C \in \mathcal{C}$. Knowing that for every $\ell \in \mathbb{N}$ there exists some cube $Q_\ell = (a_\ell,b_\ell)^2$ with $\mu_{C_\ell}(Q_\ell) = 0$ and $\lambda_2(Q_\ell) \geq \frac{1}{N}$, applying Bolzano-Weierstrass yields
     the existence of subsequences $(a_{\ell_k})_{k \in \mathbb{N}}$, $(b_{\ell_k})_{k \in \mathbb{N}}$ and $a_0,b_0 \in \mathbb{I}$ such that $a_{\ell_k} \overset{k \rightarrow \infty}{\longrightarrow} a_0$ and $b_{\ell_k} \overset{k \rightarrow \infty}{\longrightarrow} b_0$. Using Lipschitz-continuity we therefore get
    \begin{align*}
    \mu_C(Q) &= C(b_0,b_0) - C(b_0,a_0) - C(a_0,b_0) + C(a_0,a_0) \\&=
    \lim_{k \rightarrow \infty} [C_{\ell_k}(b_{\ell_k},b_{\ell_k}) - C_{\ell_k}(b_{\ell_k},a_{\ell_k}) - C_{\ell_k}(a_{\ell_k},b_{\ell_k}) + C_{\ell_k}(a_{\ell_k},a_{\ell_k})] \\&=
    \lim_{k \rightarrow \infty}\mu_{C_{\ell_k}}(Q_{\ell_k}) = 0.
    \end{align*}
    Hence, considering $\frac{1}{N} \leq \lim_{k \rightarrow \infty} \lambda_2(Q_{\ell_k}) = \lim_{k \rightarrow \infty} (b_{\ell_k} - a_{\ell_k})^2 = (b_0-a_0)^2=
    \lambda_2(Q)$ finally implies that $\mathcal{A}_N$ is closed.\\ 
    Taking an arbitrary copula $C \in \mathcal{C}$ and fixing $B \in \mathcal{C}$ with $\mathrm{supp}(\mu_B) = \mathbb{I}^2$, defining $C_n$ by $C_n := (1-\frac{1}{n})C + \frac{1}{n}B$, obviously 
    $\mathrm{supp}(\mu_{C_n}) = \mathbb{I}^2$ and $d_\infty(C,C_n) \overset{n \rightarrow \infty} \longrightarrow 0$.
    This shows that the family of copulas with full support is dense in $\mathcal{C}$, implying that
     $\mathcal{A}_N$ is nowhere dense for every $N$. 
     Finally, using 
     $$\{C \in \mathcal{C} \colon \mathrm{supp}(\mu_C) \neq \mathbb{I}^2\} \subseteq \bigcup_{N \in \mathbb{N}}\mathcal{A}_N$$ yields that $\{C \in \mathcal{C} \colon \mathrm{supp}(\mu_C) = \mathbb{I}^2\}$ is co-meager in $(\mathcal{C}, d_\infty)$.
\end{proof}
We close this section with the following quite counter-intuitive result, which directly follows from  
combining Theorem \ref{thm:kim} and Lemma \ref{lem:typical:cop_full_supp} and using the fact that 
finite and countably infinite intersections of co-meager sets are co-meager as well.
\begin{Cor}\label{cor:typical_cop_full_mcd_full_support}
    A typical copula $C \in (\mathcal{C}, d_\infty)$ is mutually completely dependent and has full support.
\end{Cor}

\clearpage
\section{Extreme Value copulas}\label{section:EVC}
\subsection{Some preliminaries}
A copula $C \in \mathcal{C}$ is called Extreme Value copula (EVC) if there exists some copula $B \in \mathcal{C}$ such that
$$
C(x,y)  = \lim_{n \rightarrow \infty} B^n(x^\frac{1}{n},y^\frac{1}{n})
$$
for all $x,y \in \mathbb{I}$. Throughout this section the space of all bivariate EVCs will be denoted by 
$\mathcal{C}_{ev}$. It is well-known (see \cite{evc-mass}) that $(\mathcal{C}_{ev},d_\infty)$ is a compact metric space. Moreover, according to \cite{dur_princ,haan1977,nelsen2006,Pickands}, the following assertions are equivalent
\begin{itemize}
    \item[(i)] $C \in \mathcal{C}_{ev}$
    \item[(ii)] $C$ is max-stable, i.e., $C(x,y) = C^k(x^\frac{1}{k},y^\frac{1}{k})$ for all $k \in \mathbb{N}$ and $x,y \in \mathbb{I}$.
    \item[(iii)] There exists a Pickands dependence function $A$, i.e., a convex function $A \colon \mathbb{I} \rightarrow \mathbb{I}$ fulfilling $\max\{1-x,x\} \leq A(x) \leq 1$ for $x \in \mathbb{I}$, such that the identity
    \begin{equation}\label{eq:map_eq_pick_copula}
    C(x,y) = (xy)^{A\left(\frac{\log(x)}{\log(xy)}\right)}
    \end{equation}
    holds for all $x,y \in (0,1)$.
\end{itemize}
The family of all Pickands dependence functions will be denoted by $\mathcal{A}$ and we will let $C_A$ denote the unique EVC induced by $A \in \mathcal{A}$. 

For every Pickands dependence function $A \in \mathcal{A}$ we will let $D^+A(x)$ denote the right-hand derivative of 
$A$ at $x \in [0,1)$ and $D^-A(x)$ the left-hand derivative of $A$ at $x \in (0,1]$.  
Convexity of $A$ implies that $D^+A(x)=D^-A(x)$ holds for all but at most countably 
many $x \in (0,1)$, i.e., $A$ is differentiable outside a countable subset of $(0,1)$, that $D^+A$ is 
non-decreasing and right-continuous on $[0,1)$ and that $D^-A$ is non-decreasing and 
left-continuous on $(0,1]$ (see \cite{Kannan1996,Pollard2001} and the references therein). 
Setting $D^+A(1):=D^-A(1)$ allows to view 
$D^+A$ as non-decreasing and right-continuous function on the full unit interval $[0,1]$, 
which, taking into account $\max\{1-x,x\} \leq A(x) \leq 1$ for all $x \in [0,1]$, 
only assumes values in $[-1,1]$. 
Additionally (again see \cite{Kannan1996,Pollard2001} and the references therein), we have $D^-A(x)=D^+A(x-)$ for every $x \in (0,1)$.

Following \cite{beirlant2004,haan1977,Pickands} every Pickands dependence function $A$ uniquely
corresponds to a spectral measure $\nu$, i.e., a measure  on the unit simplex $\Delta_2 =
\{(x, 1 - x) \colon x \in \mathbb{I}\}$ fulfilling
$$
\int_{\Lambda_2} x \mathrm{d}\nu = \int_{\Lambda_2} y \mathrm{d}\nu = 1.
$$
Projecting $\Delta_2$ onto $\mathbb{I}$ to a normalization we can identify $\nu$ with a probability measure 
$\vartheta$ on $\mathcal{B}(\mathbb{I})$ with expected value $\frac{1}{2}$.
Throughout this section we call such a measure $\vartheta$ a Pickands dependence measure and define the family of all Pickands dependence measures by
$$
\mathcal{P}_\mathcal{A} := \left\{\vartheta \in \mathcal{P}(\mathbb{I}) \colon \int_\mathbb{I}x \mathrm{d}\vartheta(x) = \frac{1}{2}\right\}.
$$
Every $\vartheta \in \mathcal{P}_\mathcal{A}$ induces a unique Pickands dependence function and vice versa, 
see \cite{evc-mass}. In fact, according to Lemma \ref{lem:measure_pickands} the mapping $\Upsilon$, defined by 
\begin{equation}\label{eq:gamma}
\Upsilon(\vartheta)(t) := 1-t + 2\int_{[0,t]}\vartheta([0,z]) \mathrm{d}\lambda(z)
\end{equation}
maps $(\mathcal{P}_\mathcal{A},\tau_w)$ to $(\mathcal{A},\Vert \cdot \Vert_\infty)$. Throughout this section we denote the family of all absolutely continuous, discrete and singular Pickands dependence measures by $\mathcal{P}_\mathcal{A}^{abs}$, $\mathcal{P}_\mathcal{A}^{dis}$ and $\mathcal{P}_\mathcal{A}^{sing}$, respectively.
Moreover, we will call a function $F\colon \mathbb{I} \rightarrow \mathbb{I}$ a distribution function/measure generating function, if its extension via setting $F(x)=0$ for $x<0$ and $F(x)=1$ for $x>1$, is a distribution function/measure generating function.\\
It is well known and straightforward to verify that the Pickands dependence function corresponding to the Fréchet-Hoeffding upper bound $M$ is defined by $A_M(t) := \max\{1-t,t\}$ and the Pickands dependence function corresponding to the independence copula $\Pi$ fulfills $A_\Pi(t) = 1$ for all $t \in \mathbb{I}$. Moreover, the measure 
$\vartheta_M = \delta_\frac{1}{2}$ is easily seen to correspond to $M$ and 
$\vartheta=\frac{1}{2}(\delta_{0} + \delta_{1})$ to $\Pi$.
Building upon the afore-mentioned characterization of EVC via Pickands dependence functions, the map 
$\Phi$, defined via equation \eqref{eq:map_eq_pick_copula} maps $(\mathcal{A},\Vert \cdot \Vert_\infty)$ to 
$(\mathcal{C}_{ev},d_\infty)$. It is straightforward to see that both $\Upsilon$ and $\Phi$ are homeomorphisms.
\begin{Lemma}\label{lem:extreme_homeom}
    The maps $\Phi$, $\Upsilon$ and $\Phi \circ \Upsilon$ are homeomorphisms.
\end{Lemma}
\begin{proof}
Obviously every $\vartheta \in \mathcal{P}_\mathcal{A}$ induces a unique $A \in \mathcal{A}$ via equation \eqref{eq:gamma}. On the other hand for every $A \in \mathcal{A}$, defining
\begin{align*}
F(t) := \begin{cases}
    \frac{D^+A(t)+1}{2},& \text{ if } t\in [0,1)\\
    1, & \text{ if } t = 1
\end{cases}
\end{align*}
it follows that $F$ induces a unique probability measure $\vartheta_A$. This measure is indeed a Pickands 
dependence measure, since we have
$$
\int_\mathbb{I} x \mathrm{d}\vartheta_A(x) = \int_\mathbb{I} (1-F(t)) \mathrm{d}\lambda(t) = 1 - \frac{1}{2}\int_\mathbb{I}D^+A(t) \mathrm{d}\lambda(t) - \frac{1}{2} = \frac{1}{2}.
$$
It is left to show that $\vartheta_A$ fulfills equation \eqref{eq:gamma}. Indeed, we obtain that
$$
\Upsilon(\vartheta_A)(t) = 1-t + 2\int_{[0,t]}\vartheta_A([0,z]) \mathrm{d}\lambda(z) = 1 + \int_{[0,t]}D^+A(z)\mathrm{d}\lambda(z) = A(t).
$$
Thus, $\Upsilon$ is a bijection and applying Lemma \ref{lem:equiv_conv_meas_evc} immediately yields that $\Upsilon$ is a homeomorphism. The fact that $\Phi$ is a homeomorphism is an immediate consequence of Lemma \ref{lem:equiv_conv_meas_evc}.
Since compositions of homeomorphisms are homeomorphisms the proof is complete.
\end{proof}
\begin{Rem}
\emph{Applying \cite[Theorem 5.1]{bernoulli}, $\Phi$ is even a homeomorphism, if we equip $\mathcal{C}_{ev}$ 
with the metric $D_1$ (or any of the metrics $D_p$).}
\end{Rem}
Given  $A \in \mathcal{A}$ define the map $G_A\colon \mathbb{I} \rightarrow \mathbb{I}$ by 
\begin{equation}\label{GA}
G_A(t) := A(t) + D^+A(t)(1-t)
\end{equation}
$t \in [0,1)$ as well as $G_A(1) := 1$. Then applying \cite[Lemma 5]{evc-mass} $G_A$ is 
non-negative, right-continuous, and non-decreasing. Hence working with $G_A$ and considering conditional distributions of an 
EVC $C_A \in \mathcal{C}_{ev}$ with corresponding Pickands dependence function $A \in \mathcal{A}$, according to \cite{evc-mass} a version of the Markov-kernel of $C_A$ is given by
\begin{align}\label{eq:ev_markov_kernel}
    K_A(x,[0,y]) := \begin{cases}
        1,& \text{ if } x \in \{0,1\}\\
        \frac{C_A(x,y)}{x}G_A\left(\frac{\log(x)}{\log(xy)}\right), &\text{ if } x,y \in (0,1)\\
        y,&\text{ if } (x,y) \in (0,1)\times \{0,1\}
    \end{cases}
\end{align}
for $x,y \in \mathbb{I}$.
Again following \cite{evc-mass}, define $f^t(x) := x^{\frac{1}{t}-1}$ for $t \in (0,1)$ and $x \in \mathbb{I}$, 
and, for a given Pickands dependence function $A \in \mathcal{A}$, set
\begin{equation}\label{eq:L and R Pickands}
L := \max\{x \in \mathbb{I} \colon A(x) = 1-x\}, \,\,\, R := \min\{x \in \mathbb{I} \colon A(x) = x\}
\end{equation}

According to \cite{evc-mass} the discrete component of $C_A$ is non-degenerated if, and only if there exists some 
$t \in (0,1)$ such that $\mu_{C_A}(\Gamma(f^t)) > 0$.
Furthermore the discrete component is fully determined by the discontinuity points of the right-hand derivative $D^+A$. 
Working with equation \eqref{eq:gamma} it is straightforward to prove that the right-hand derivative $D^+A$ of a 
Pickands dependence function $A \in \mathcal{A}$ can easily be expressed in terms of the corresponding Pickands dependence measure $\vartheta \in \mathcal{P}_\mathcal{A}$ - the following result holds:
\begin{Lemma}\label{lem:der_pick_meas}
    Let $\vartheta \in \mathcal{P}_\mathcal{A}$ and $A \in \mathcal{A}$ denote the corresponding Pickands dependence function according to equation \eqref{eq:gamma}. Then
    \begin{equation}\label{eq:der_pick_meas}
        D^+A(t) = 2\vartheta([0,t])-1
    \end{equation}
   holds for every $t \in [0,1)$.
\end{Lemma}
\begin{proof}
Defining $F_\vartheta(t) := \vartheta([0,t])$ for $t \in \mathbb{I}$ and fixing an 
arbitrary continuity point $t$ of $F_\vartheta$, applying equation \eqref{eq:gamma} and Lebesgue's differentiation theorem (see \cite{rudin1974}) we obtain that
$$
D^+A(t) = \lim_{h \downarrow 0}2\int_{[t,t+h]}\vartheta([0,z]) \mathrm{d}\lambda(z) - 1 = 2\vartheta([0,t]) - 1.
$$
Using the fact that the set $\mathrm{Cont}(F_\vartheta)$ of all continuity points of $F_\vartheta$ 
is dense in $\mathbb{I}$ and considering that both $D^+A$ and $F_\vartheta$ are right-continuous yields the result.
\end{proof}
The following Theorem shows that $\Gamma(f^t)$ with $t \in (0,1)$ carries mass if, and only if, $t$ is a point mass 
of $\vartheta$.
\begin{theorem}\label{lem:mass_graph}
    Let $C \in \mathcal{C}_{ev}$ be an EVC with associated Pickands dependence 
    measure $\vartheta \in \mathcal{P}_\mathcal{A}$. Then the following identity holds for all $t \in (0,1)$:
    \begin{equation}\label{eq:mass_graph}
    \mu_{C}(\Gamma(f^t)) =  \frac{2t(1-t)}{A(t)}\vartheta(\{t\}).
    \end{equation}
    Consequently, $\vartheta$ has a point mass if, and only if $\mu_C^{dis}(\mathbb{I}^2) > 0$.
\end{theorem}
\begin{proof}
    Applying Lemma \ref{lem:der_pick_meas} it follows that $\vartheta$ has a point mass $t \in (0,1)$ if, and only if 
    $t \in (0,1)$ is a point of discontinuity of $D^+A$. The assertion now follows by applying \cite[Lemma 4]{evc-mass}. Equation \eqref{eq:mass_graph} is an immediate consequence of \cite[Lemma 4]{evc-mass} and Lemma \ref{lem:der_pick_meas}.
\end{proof}
\subsection{Mass distributions of Extreme Value copulas}
In what follows we characterize the existence of a non-degenerated absolutely continuous, discrete, or singular component 
of an EVC in terms of the corresponding Pickands dependence measures. We do not consider purely discrete or 
singular Pickands dependence measures since, as proved already in \cite[Corollary 5]{evc-mass}, every EVC (except from $M$) has a non-degenerated absolutely continuous component and thus, purely discrete and singular EVC do not exist.
We start with the following example illustrating why working with $(0,1)$ instead of $\mathbb{I}$ can't be avoided.
\begin{Ex}
Consider the Pickands dependence measure $\vartheta_1=\frac{1}{2}(\delta_{0} + \delta_{1})$. 
As mentioned before, $\vartheta_1$ induces the product copula $\Pi$, i.e., even though the EVC copula is absolutely
continuous, the Pickands dependence measure is not. It is straightforward to see that $\Pi$ is not the only 
absolutely continuous EVC whose associated Pickands dependence measure has a non-degenerated discrete component. 
In fact, letting $\vartheta_2$ be an arbitrary absolutely continuous measure in $\mathcal{P}_\mathcal{A}$ 
and defining $\vartheta := \frac{\vartheta_1 + \vartheta_2}{2} \in \mathcal{P}_\mathcal{A}$, then 
$\vartheta$ is a probability measure with non-degenerated discrete component which induces an absolutely continuous EVC.
\end{Ex}
We now characterize the existence of non-degenerated discrete/singular components of $C_A$ in terms of the Lebesgue decomposition of the Pickands dependence measure $\vartheta$ (the quite technical proof as well as 
some preliminary lemmas can be found in the Appendix).
\begin{theorem}\label{thm:evc_regularity_meas_cop}
Let $C \in \mathcal{C}_{ev}$ be an EVC and let $\vartheta \in \mathcal{P}_\mathcal{A}$ and 
$A \in \mathcal{A}$ denote the corresponding Pickands dependence measure and Pickands dependence function, respectively. 
Then the following equivalences hold:
\begin{itemize}
    \item[(i)] $\vartheta$ is absolutely continuous on $(0,1)$ if, and only if $C$ is absolutely continuous.
    \item[(ii)] $\vartheta$ has a point mass in $(0,1)$ if, and only if $C$ has non-degenerated discrete component.
    \item[(iii)] $\vartheta$ has a non-degenerated singular component if, and only if $C$ has non-degenerated singular component.
\end{itemize}
\end{theorem}

The following example illustrates the previous result and considers a Pickands dependence measure 
with non-degenerated discrete and absolutely continuous component.
\begin{Ex}\label{ex:dis_and_abs_pickands}
    Define the distribution function $F_\vartheta$ (see Figure \ref{fig:dis_abs_meas_pickands}) by
    $$
    F_\vartheta(t) := \begin{cases}
      t,& \text{ if } t \in [0,\frac{1}{2})\\
      \frac{3}{5},& \text{ if } t \in [\frac{1}{2},\frac{3}{4})\\
      \frac{16}{35}t + \frac{1}{2} ,& \text{ if } t \in [\frac{3}{4},1)\\
    1 ,& \text{ if } t = 1
    \end{cases}
    $$
    for $t \in \mathbb{I}$. It is straightforward to verify that $F_\vartheta$ corresponds to a unique Pickands dependence measure $\vartheta \in \mathcal{P}_\mathcal{A}$, which has both a non-degenerated discrete and 
    absolutely continuous component. According to equation \eqref{eq:gamma} the associated Pickands dependence function $A$ is given by (again see Figure \ref{fig:dis_abs_meas_pickands})
    $$
    A(t) =\begin{cases}
      t^2-t+1,& \text{ if } t \in [0,\frac{1}{2})\\
      \frac{1}{5}t + \frac{13}{20},& \text{ if } t \in [\frac{1}{2},\frac{3}{4})\\
      \frac{16}{35}t^2 + \frac{1}{2} ,& \text{ if } t \in [\frac{3}{4},1].
    \end{cases}
    $$
    Applying Theorem \ref{thm:evc_regularity_meas_cop}, both the discrete and absolutely continuous component of $\vartheta$ propagate to the induced copula $C_A$. The discrete component is concentrated on the graphs of the functions $f^\frac{1}{2}(x) = x$ and $f^\frac{1}{3}(x) = \sqrt[3]{x}$, respectively. A sample of the copula $C_A$ is depicted in Figure \ref{fig:discrete_abs_pickands_sample}.
\end{Ex}
\begin{figure}[!ht]
	\centering
	\includegraphics[width=1\textwidth]{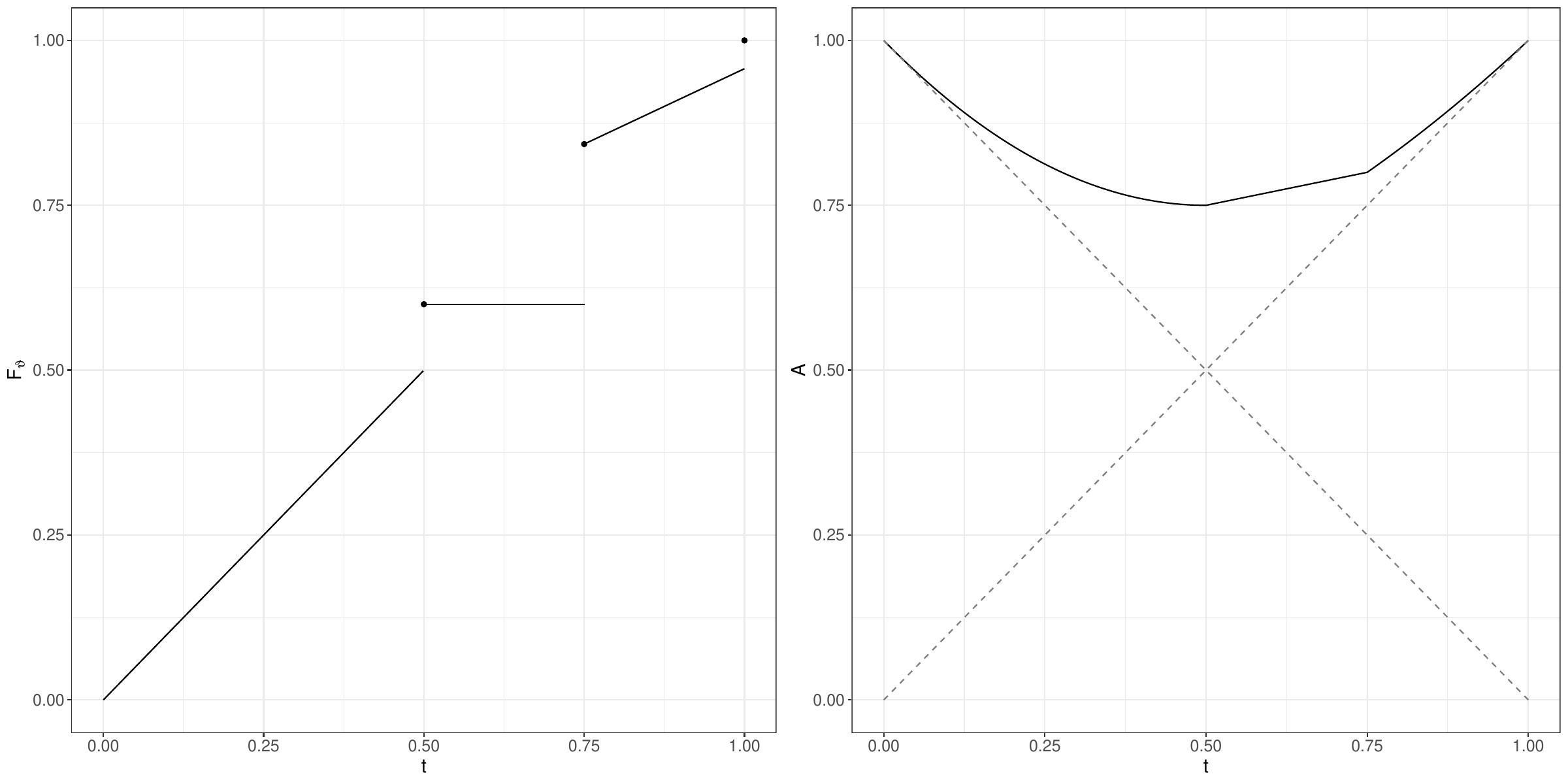}
	\caption{Graphs of the distribution function $F_\vartheta$ of the Pickands dependence measure $\vartheta$ (left) according to Example \ref{ex:dis_and_abs_pickands} and associated Pickands dependence function $A$ (right) associated with it.}
    \label{fig:dis_abs_meas_pickands}
    \end{figure}
    \begin{figure}[!ht]
	\centering
	\includegraphics[width=1.1\textwidth]{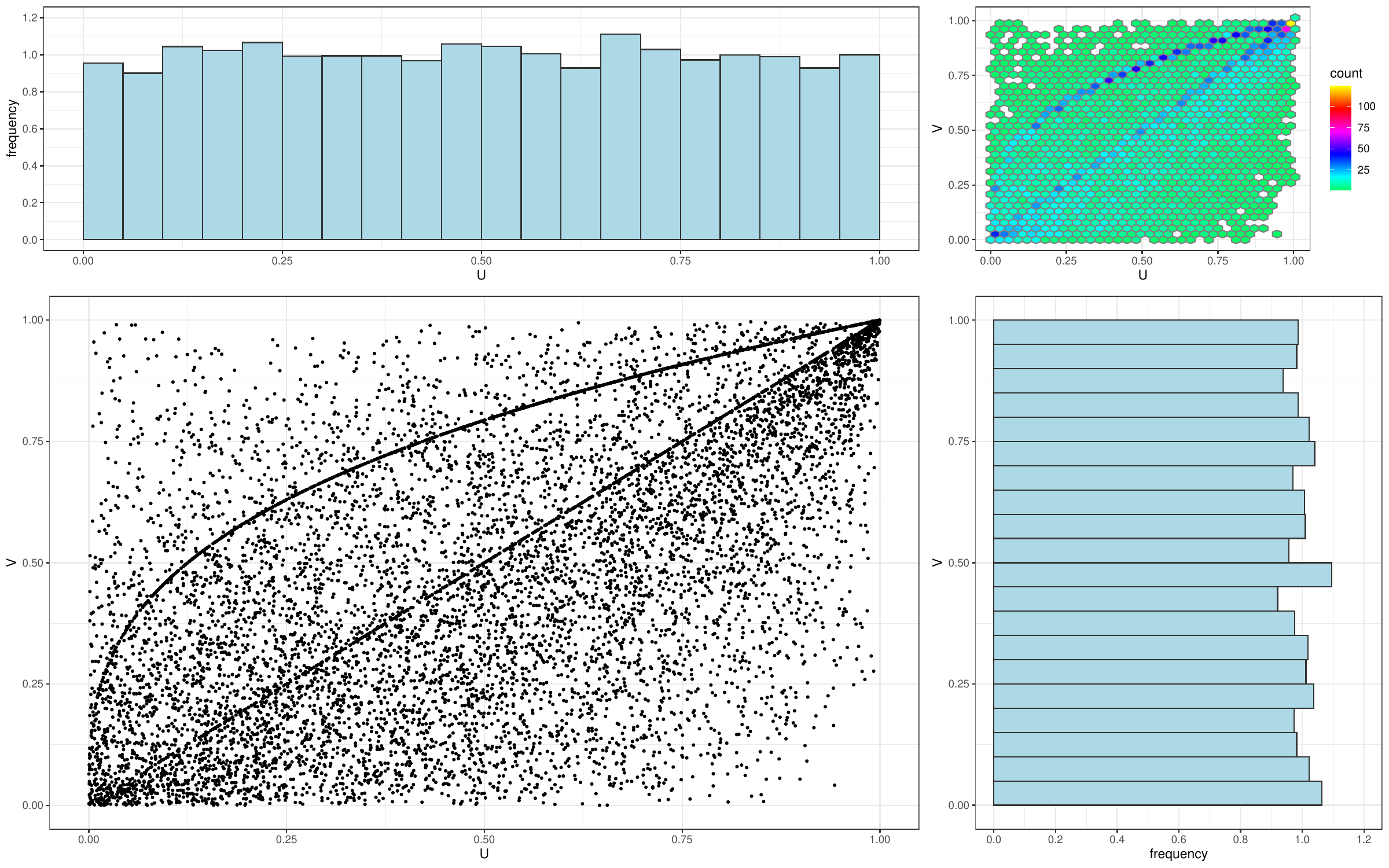}
	\caption{Sample of size 10000 of the EVC $C_A$, where $A$ is the Pickands dependence function according to Example \ref{ex:dis_and_abs_pickands}, its histogram and the two marginal histograms. The sample has been generated via conditional inverse sampling.}
\label{fig:discrete_abs_pickands_sample}
\end{figure}
We have already shown that many regularity/singularity properties of the Pickands dependence measure carry over to the corresponding EVC. As next step we show that the support is no exception:
\begin{Lemma}\label{lem:support_measure_copula}
Let $C \in \mathcal{C}_{ev}$ and $\vartheta \in \mathcal{P}_\mathcal{A}$ be its associated Pickands dependence measure. Then the following assertions hold:
\begin{itemize}
    \item[(i)] If $\vartheta^{dis}$ has full support, then $\mathrm{supp}(\mu_C^{dis}) = \mathbb{I}^2$.
    \item[(ii)] If $\vartheta^{sing}$ has full support, then $\mathrm{supp}(\mu_C^{sing}) = \mathbb{I}^2$.
    \item[(iii)] If $\vartheta$ fulfills $L = 0$ and $R = 1$, then $\mathrm{supp}(\mu_C^{abs}) = \mathbb{I}^2$.
\end{itemize}
\end{Lemma}
\begin{proof}
Suppose that $\vartheta \in \mathcal{P}_\mathcal{A}$ 
is such that $\vartheta^{dis}$ has full support. Then there exist countably infinite sets
 $\{q_1,q_2,...\} \subseteq \mathbb{I}$ and $\{\alpha_1,\alpha_2,...\} \subseteq (0,1)$ with $\sum_{i \in \mathbb{N}} \alpha_i = 1$ such that $\vartheta^{dis} = \sum_{i \in \mathbb{N}} \alpha_i\delta_{q_i}$. Fixing $x \in (0,1)$, applying Theorem \ref{lem:mass_graph} yields $K_A(x,\{f_t(x)\}) > 0$ if, and only if $t = q_i$ for some $i \in \mathbb{N}$. 
Having this it follows immediately that $\textrm{supp}(K_C^{dis}(x,\cdot)) = \mathbb{I}$. 
Since $x \in (0,1)$ was arbitrary, applying disintegration yields $\mathrm{supp}(\mu_C^{dis}) = \mathbb{I}^2$.\\
To prove the second assertion we may proceed as follows: The function $y \mapsto G_A^{sing}\left(\frac{\log(x)}{\log(xy)}\right)$ is continuous, singular and strictly increasing. On the other hand, the function $y \mapsto C(x,y)/x$ is absolutely continuous and strictly increasing so, applying the same arguments as in Lemma \ref{lem:help_regularity_measure}, it follows that $K_C^{sign}(x,\cdot)$ has support $\mathbb{I}$ and thus, disintegration yields the desired statement.\\
The last assertion is a direct consequence of \cite[Corollary 5]{evc-mass}.
\end{proof}
Finally, we prove that the family of EVCs $C \in \mathcal{C}_{ev}$ with $\mathrm{supp}(\mu_C^{dis}) = \mathbb{I}^2$, the family fulfilling $\mathrm{supp}(\mu_C^{sing}) = \mathbb{I}^2$ and the family with  $\mathrm{supp}(\mu_C^{abs}) = \mathbb{I}^2$ is dense in $\mathcal{C}_{ev}$.
\begin{theorem}\label{thm:evc_density_result}
The following assertions hold:
\begin{itemize}
    \item[(i)] The family $\{C\in \mathcal{C}_{ev}\colon\mathrm{supp}(\mu_C^{dis})=\mathbb{I}^2\}$ is dense in $(\mathcal{C}_{ev},d_\infty)$.
    \item[(ii)] The family $\{C\in \mathcal{C}_{ev}\colon\mathrm{supp}(\mu_C^{sing})=\mathbb{I}^2\}$ is dense in $(\mathcal{C}_{ev},d_\infty)$.
    \item[(iii)] The family $\{C\in \mathcal{C}_{ev}\colon\mathrm{supp}(\mu_C^{abs})=\mathbb{I}^2\}$ is dense in $(\mathcal{C}_{ev},d_\infty)$.
\end{itemize}
\end{theorem}
\begin{proof}
Combining Lemma \ref{lem:support_measure_copula}, Lemma \ref{lem:equiv_conv_meas_evc} and Lemma \ref{lem:ev-measure-approximaton} yields the first two assertions. The last assertion follows by combining assertion $(ii)$ and the fact that $\mathrm{supp}(\mu_C) = \mathrm{supp}(\mu_C^{abs})$ holds for every $C \in \mathcal{C}_{ev}^2$ (see \cite[Corollary 5]{evc-mass}).
\end{proof}
Working with convex-combinations of Pickands dependence measures we finally obtain the following striking result:
\begin{theorem}\label{thm:evc_density_convex_comb}
The family $\{C\in \mathcal{C}_{ev}\colon \mathrm{supp}(\mu_C^{dis})= \mathrm{supp}(\mu_C^{sing}) = \mathrm{supp}(\mu_C^{abs}) = \mathbb{I}^2\}$ is dense in $(\mathcal{C}_{ev},d_\infty)$.
\end{theorem}
The afore-mentioned results remain valid when working with stronger notions of convergence - using Theorem \ref{thm:evc_density_result} and Theorem \ref{thm:evc_density_convex_comb} in combination with \cite[Theorem 5.1]{bernoulli} yields the following two corollaries:
\begin{Cor}\label{cor:evc_density_result_wcc}
The following assertions hold:
\begin{itemize}
    \item[(i)] The family $\{C\in \mathcal{C}_{ev}\colon\mathrm{supp}(\mu_C^{dis})=\mathbb{I}^2\}$ is dense in $(\mathcal{C}_{ev},D_1)$ and in $\mathcal{C}_{ev}$ w.r.t. wcc.
    \item[(ii)] The family $\{C\in \mathcal{C}_{ev}\colon\mathrm{supp}(\mu_C^{sing})=\mathbb{I}^2\}$ is dense in $(\mathcal{C}_{ev},D_1)$ and in $\mathcal{C}_{ev}$ w.r.t. wcc.
    \item[(iii)] The family $\{C\in \mathcal{C}_{ev}\colon\mathrm{supp}(\mu_C^{abs})=\mathbb{I}^2\}$ is dense in $(\mathcal{C}_{ev},D_1)$ and in $\mathcal{C}_{ev}$ w.r.t. wcc.
\end{itemize}
\end{Cor}
\begin{Cor}
The family $\{C\in \mathcal{C}_{ev}\colon \mathrm{supp}(\mu_C^{dis})= \mathrm{supp}(\mu_C^{sing}) = \mathrm{supp}(\mu_C^{abs}) = \mathbb{I}^2\}$ is dense in $(\mathcal{C}_{ev},D_1)$ and in $\mathcal{C}_{ev}$ w.r.t. wcc.
\end{Cor}
\subsection{Derivatives of Extreme Value copulas}
We now revisit the results from Section \ref{section:der_copulas} and show that EVC typically exhibit more regular behavior.
To simplify notation we denote the family of all EVC $C \in \mathcal{C}_{ev}$ with the property that 
there exists some set $\Lambda \in \mathcal{B}(\mathbb{I})$ with $\lambda(\Lambda) = 1$ such that for all 
$x \in \Lambda$ there exists some $y = y_x \in (0,1)$ such that $\partial_1C(x,y)$ does not exist 
by $\mathcal{C}_{ev,p} := \mathcal{C}_{p} \cap \mathcal{C}_{ev}$. The more pathological family $\mathcal{C}_{ev,\mathcal{Q}}$
is defined in the same manner, i.e., $\mathcal{C}_{ev,\mathcal{Q}} :=  \mathcal{C}_{\mathcal{Q}} \cap \mathcal{C}_{ev}$. 
In other words: For every $C \in \mathcal{C}_{ev,\mathcal{Q}}$ there exists some set 
$\Lambda \in \mathcal{B}(\mathbb{I})$ with $\lambda(\Lambda) = 1$ such that for every $x \in \Lambda$
the partial derivative $\partial_1C(x,y)$ does not exist on a dense set of $y \in \mathbb{I}$. \\
 We start by showing that $\mathcal{C}_{ev,\mathcal{Q}}$ is non-empty.
\begin{Ex}[Example of an element of $\mathcal{C}_{ev,\mathcal{Q}}$]\label{ex:evc_nondiff}
Suppose that $\mathcal{Q}=\{q_1,q_2,...\} \subset (0,1)$ is dense in $\mathbb{I}$. Furthermore suppose that $\{\alpha_1,\alpha_2,\ldots\} 
\subset (0,1)$ fulfills $\sum_{i \in \mathbb{N}}\alpha_i = 1$. 
Then defining $\vartheta \in \mathcal{P}(\mathbb{I})$ by
$$
\tilde{\vartheta} = \sum_{i = 1}^\infty\alpha_i\delta_{q_i}
$$
and, if necessary, normalizing $\tilde{\vartheta}$ in the sense of Lemma \ref{lem:normalizing_pickands} yields 
a discrete measure $\vartheta \in \mathcal{P}_\mathcal{A}$ with full support $\mathbb{I}$. For the sake of simplicity
we will assume that $\vartheta=\tilde{\vartheta}$ holds. 
Let $A=A_\vartheta \in \mathcal{A}$ denote the Pickands dependence function induced by $\vartheta$ and $C_A$ 
the corresponding EVC. 
We want to show that for every $x_0 \in (0,1)$ we have that $x \mapsto C_A(x,f^{q_j}(x_0))$ is not differentiable 
at $x_0$ and proceed as follows:
Let $x_0 \in (0,1)$ be arbitrary but fixed, set $y_j := x_0^{\frac{1}{q_j}-1}$ with $q_j \in \mathcal{Q}$ 
and, using equation \eqref{eq:der_pick_meas}, consider the function
\begin{align*}
x \mapsto D^{+}A\left(\frac{\log(x)}{\log(xy_j)}\right) &= 2\left(\underbrace{\sum_{i \neq j}\alpha_i\mathbf{1}_{[0,\frac{\log(x)}{\log(xy_j)}]}(q_i)}_{=: I(x)} + \underbrace{\alpha_j\mathbf{1}_{[0,\frac{\log(x)}{\log(xy_j)}]}(q_j)}_{=: II(x)}\right) - 1.
\end{align*}
Applying the fact that $\frac{\log(x_0)}{\log(x_0y_j)} = q_j \neq q_i$ for $i \neq j$ yields that $I(x)$ is continuous in $x_0$. Defining the functions
$$
\xi(s,y_j) := K_A(s,[0,y_j]) - 2C_A(s,y_j)\alpha_j\mathbf{1}_{[0,\frac{\log(s)}{\log(sy_j)}]}(q_j)\frac{\log(y_j)}{s\log(sy_j)},
$$
$$
\varphi(x,y_j) := 2\int_{[0,x]}C_A(s,y_j)\alpha_j\mathbf{1}_{[0,\frac{\log(s)}{\log(sy_j)}]}(q_j)\frac{\log(y_j)}{s\log(sy_j)}\mathrm{d}\lambda(s),
$$
$$
\psi(x,y_j) := \int_{[0,x]}\xi(s,y_j)\mathrm{d}\lambda(s),
$$
for every $s,x \in (0,1)$ and applying disintegration we can write the copula $C_A$ as
$$
C_A(x,y_j) = \varphi(x,y_j) + \psi(x,y_j) 
$$
for every $x \in (0,1)$. Furthermore, using that $x_0$ is a point of continuity of the function $x \mapsto \xi(x,y_j)$
implies that the function $x \mapsto \psi(x,y_j)$ is differentiable at $x_0$.
Moreover, working with the right- and left-hand derivatives of $\varphi$, respectively, we obtain that
\begin{align*}
\partial_1^+\varphi(x_0,y_j) &= 2\lim_{h \downarrow 0}\frac{1}{h}\int_{[x_0,x_0+h]}C_A(s,y_j)\alpha_j\mathbf{1}_{[0,\frac{\log(s)}{\log(sy_j)}]}(q_j)\frac{\log(y_j)}{s\log(sy_j)}\mathrm{d}\lambda(s) \\&= 2\alpha_jC_A(x_0,y_j) \frac{\log(y_j)}{x_0\log(x_0y_j)}
\end{align*}
as well as
\begin{align*}
\partial_1^-\varphi(x_0,y_j) &= \lim_{h \downarrow 0} \frac{1}{h}\int_{[x_0-h,x_0]}C_A(s,y_j)\alpha_j\mathbf{1}_{[0,\frac{\log(s)}{\log(sy_j)}]}(q_j)\frac{\log(y_j)}{s\log(sy_j)}\mathrm{d}\lambda(s) = 0,
\end{align*}
which altogether shows that $x \mapsto C_A(\cdot,y_j)$ is not differentiable in $x_0$. Since $y_j$ was arbitrary 
and the set $\{x_0^{1/q_j-1}: j \in \mathbb{N}\}$ is obviously dense in $\mathbb{I}$ this completes the proof.
\end{Ex}
Similarly to Theorem \ref{thm:pathological_dense}, EVCs exhibiting pathological are well spread - the following 
\begin{Cor}\label{cor:CevQdense}
    The set $\mathcal{C}_{ev,\mathcal{Q}}$ is dense in $(\mathcal{C}_{ev},d_\infty)$.
\end{Cor}
\begin{proof}
Let $C_A \in \mathcal{C}_{ev}$ and $\vartheta \in \mathcal{P}_\mathcal{A}$ be its Pickands dependence measure. 
Then, using Lemma \ref{lem:normalizing_pickands} there exists a 
sequence $(\vartheta_n)_{n \in \mathbb{N}}$ in $\mathcal{P}_\mathcal{A}$ consisting of discrete measures with full 
support that converges weakly to $\vartheta$. Applying Lemma \ref{lem:equiv_conv_meas_evc} now yields the result.
\end{proof}
\noindent The following simple observation will be key for studying Baire category results for EVC in the next section.
\begin{theorem}\label{thm:equiv_discrete_path_ev}
    Let $C \in \mathcal{C}_{ev}$. Then $\mu_C^{dis}(\mathbb{I}^2) > 0$ if, and only if $C\in \mathcal{C}_{ev,p}$.
\end{theorem}
\begin{proof}
In case of $\mu_C^{dis}(\mathbb{I}^2) > 0$ applying Theorem \ref{thm:evc_regularity_meas_cop} yields that 
$\vartheta \in \mathcal{P}_\mathcal{A}$ has a point mass in some point $t_0 \in (0,1)$. 
For fixed $x \in (0,1)$, setting $y_x = f^{t_0}(x)$ and proceeding analogously to Example \ref{ex:evc_nondiff} shows that $\partial_1^+C(x,y_x) \neq \partial_1^-C(x,y_x)$.\\
For the other direction suppose that $C \in \mathcal{C}_{ev,p}$. Then there exists some set $\Lambda \in \mathcal{B}(\mathbb{I})$ with $\lambda(\Lambda) = 1$ such that for $x \in \Lambda$ we can find some $y_x \in (0,1)$ 
with $\partial_1^+C(x,y_x) \neq \partial_1^-C(x,y_x)$. It follows that $D^+A(\frac{\log(x)}{\log(xy_x)}) \neq D^-A(\frac{\log(x)}{\log(xy_x)})$, so $\frac{\log(x)}{\log(xy_x)}$ is a point of discontinuity of $D^+A$, and 
applying Theorem \ref{lem:mass_graph} yields $\mu_C^{dis}(\mathbb{I}^2) > 0$.
\end{proof}
\subsection{Baire category results for Extreme Value copulas}
Throughout this section we again work with the afore-mentioned one-to-one-to-one 
correspondence between Pickands dependence functions $A \in \mathcal{A}$, Pickands dependence measures 
$\vartheta \in \mathcal{P}_\mathcal{A}$, and EVC $C \in \mathcal{C}_{ev}$. 
Rewriting equation \eqref{eq:L and R Pickands} in terms of the measure $\vartheta \in \mathcal{P}_\mathcal{A}$ obviously yields
\begin{equation}\label{eq:L and R measure}
L_\vartheta = \sup\{x \in \mathbb{I}\colon \vartheta([0,x]) = 0\}, \,\,\, R_\vartheta = \inf\{x \in \mathbb{I}\colon \vartheta([0,x]) = 1\},
\end{equation}
with the convention $\sup \emptyset:=0$ and $\inf \emptyset:=1$.
According to \cite{evc-mass} the support of an EVC $C \in \mathcal{C}_{ev}$ is fully determined by the functions $f^{L_\vartheta}$ and $f^{L_\vartheta}$, respectively, i.e.,
$$
\mathrm{supp}(\mu_{C}) = \{(x,y)\in \mathbb{I}^2 \colon f^{L_\vartheta}(x) \leq y \leq f^{R_\vartheta}(x)\}.
$$
In other words: An EVC $C \in \mathcal{C}_{ev}$ has full support if, and only if $L_\vartheta = 0$ and $R_\vartheta = 1$. We will show now that a topologically typical Pickands dependence measure has this property:
\begin{Lemma}
    The set
    $$
    \{\vartheta \in \mathcal{P}_\mathcal{A} \colon L_\vartheta = 0 \text{ and } R_\vartheta = 1\}
    $$
    is co-meager in $\mathcal{P}_\mathcal{A}$ w.r.t the weak topology.
\end{Lemma}
\begin{proof}
    It suffices to shows that the set 
    $\{\vartheta \in \mathcal{P}_\mathcal{A} \colon L_\vartheta = 0 \text{ and } R_\vartheta = 1\}^c$ is of first 
    category, which can be done as follows. For every $n \in \mathbb{N}$ define the sets $A_n$, $B_n$ by
    $$
    A_n := \{\vartheta \in \mathcal{P}_\mathcal{A} \colon  \vartheta((0,\tfrac{1}{n})) = 0\} \text{ and }  B_n := \{\vartheta \in \mathcal{P}_\mathcal{A} \colon  \vartheta([0,1-\tfrac{1}{n}]) = 1\}.
    $$ 
    Suppose that $\vartheta,\vartheta_1,\vartheta_2,... \in A_n$ are such that
     $(\vartheta_\ell)_{\ell \in \mathbb{N}}$ converges weakly to $\vartheta$. Then applying Portmanteau's theorem 
     yields $\vartheta \in A_n$, so the set $A_n$ is closed with respect to the weak topology. 
     Since an analogous argument shows that each set $B_n$ is closed it follows that $A_n \cup B_n$ is closed. 
     Since measures with strictly increasing distribution functions are dense in $\mathcal{P}_\mathcal{A}$
     (see the proof of Corollary \ref{cor:CevQdense}), the sets $A_n \cup B_n$ are nowhere dense in $\mathcal{P}_\mathcal{A}$. Considering that $\{\vartheta \in \mathcal{P}_\mathcal{A} \colon L_\vartheta>0 \text{ or } R_\vartheta < 1\} \subseteq \bigcup_{n \in \mathbb{N}}(A_n \cup B_n)$, the desired result follows. 
\end{proof}
Applying Lemma \ref{lem:extreme_homeom} together with the previous lemma yields the following corollary, stating 
that typical EVCs have full support.
\begin{Cor}\label{cor:evc_full_sup}
    The set 
    $$
    \{C \in \mathcal{C}_{ev} \colon \mathrm{supp}(\mu_C) = \mathbb{I}^2\}
    $$
    is co-meager in $(\mathcal{C}_{ev},d_\infty)$. In other words: Topologically typical EVCs have full support.
\end{Cor}
In analogy with the space of all bivariate copulas $\mathcal{C}$ (see \cite{typ_cop_sing}), typical EVCs are not absolutely continuous. Our proofs builds upon Lemma \ref{lem:extreme_homeom}.
\begin{Lemma}
The set of all absolutely continuous Pickands dependence measures $\mathcal{P}_\mathcal{A}^{abs}$ is of first Baire category in $\mathcal{P}_\mathcal{A}$ with respect to the weak topology.
\end{Lemma}
\begin{proof}
For every $\vartheta \in \mathcal{P}_{\mathcal{A}}^{abs}$ let $\mathcal{k}_\vartheta$ denote its density and, for 
every $n \in \mathbb{N}$ define the set $G_n^\vartheta$ by
$$
G_n^\vartheta := \{x \in \mathbb{I} \colon \mathcal{k}_\vartheta(x) > n\}. 
$$
Furthermore set 
$$
\mathcal{W}_n := \left\{\vartheta \in \mathcal{P}_{\mathcal{A}}^{abs}\colon \vartheta(G_n^\vartheta) \leq \frac{1}{4}\right\}.
$$
Considering $\lambda(\bigcap_{n=1}^\infty G_n^\vartheta)=0$, absolute continuity implies 
 $\vartheta(\bigcap_{n=1}^\infty G_n^\vartheta)=0$ so, using continuity from above we get that  
 $\vartheta(G_n^\vartheta)< \frac{1}{4}$ holds for all $n$ sufficiently large. 
This implies $$\mathcal{P}_{\mathcal{A}}^{abs} \subseteq \bigcup_{n \in \mathbb{N}}\mathcal{W}_n.$$ 
To complete the proof it suffices to show that $\mathcal{W}_n$ is nowhere dense in $(\mathcal{P}_{\mathcal{A}},\tau_w)$, which 
can be done as follows: Consider an arbitrary discrete Pickands dependence measure $\beta \in \mathcal{P}_{\mathcal{A}}$
with only finitely many point masses, i.e., $\beta = \sum_{i = 1}^N \alpha_i \delta_{x_i}$ with $x_i \in (0,1)$, 
$\alpha_i \in (0,1]$, $2 \leq N \in \mathbb{N}$ and $\sum_{i=1}^N \alpha_i = 1$. 
Setting $x_0 := 0, x_{N+1} := 1$ and considering 
$$
r := \frac{1}{8nN} \min\left\{|x_{i}-x_j|:i,j, \in \{0,\ldots,N+1\}\right\}
$$
obviously $\lambda(\bigcup_{i=1}^N (x_i-r,x_i+r)) \leq \frac{1}{4n}$ holds.
Letting $f: \mathbb{I} \rightarrow [0,\infty)$ denote a continuous function fulfilling 
$f(x_i) = 1$, $f\mid_{(x_i-r,x_i+r)} \in (0,1]$ and $f = 0$ in $\mathbb{I} \setminus \bigcup_{i=1}^N (x_i-r,x_i+r)$, 
then $\int_{\mathbb{I}} f \mathrm{d}\beta = 1$ by construction. On the other hand, for arbitrary 
$\vartheta \in \mathcal{W}_n$ we have that
\begin{align*}
\int_{\mathbb{I}}f\mathrm{d}\vartheta &= \int_{G_n^\vartheta}f\mathrm{d}\vartheta + 
\int_{\mathbb{I} \setminus G_n^\vartheta}f\mathrm{d}\vartheta \leq 
\frac{1}{4} + \int_{\mathbb{I} \setminus G_n^\vartheta}f\mathrm{d}\vartheta  \leq \frac{1}{2} < 
1=\int_{\mathbb{I}} f \mathrm{d}\beta = 1.
\end{align*}
Using the fact, that (as direct consequence of Lemma \ref{lem:normalizing_pickands} in combination 
with Glivenko-Cantelli's theorem) 
discrete Pickands dependence measures with finitely many point masses are dense 
in $(\mathcal{P}_{\mathcal{A}},\tau_w)$ yields that $\mathcal{W}_n$ is nowhere dense in $\mathcal{P}_\mathcal{A}$, 
which completes the proof.
\end{proof}
\noindent As direct consequence we get the following result on the family of absolutely continuous EVCs.
\begin{theorem}\label{cor:abs_ev_meager}
    The family $\mathcal{C}_{ev,abs}$ is of first category in $(\mathcal{C}_{ev},d_\infty)$.
\end{theorem}
We now return to differentiability and show that typical Pickands dependence functions are everywhere 
differentiable.
\begin{Lemma}\label{lem:pickands_co_meager}
The set
$$\{A \in \mathcal{A}\colon A \text{ is differentiable at every } x \in (0,1)\}$$
is co-meager in $\mathcal{A}$ with respect to the uniform distance $\Vert \cdot \Vert_\infty$ on $\mathcal{A}$.
\end{Lemma}
\begin{proof}
    We prove that the set
    $$
    \hat{\mathcal{A}} := \{A \in \mathcal{A}\colon \exists x \in (0,1) \text{ such that } D^+A(x) > D^{-}A(x)\}
    $$
    is of first Baire category with respect to the topology induced by $\Vert \cdot \Vert_\infty$. For arbitrary $k \in \mathbb{N}$ and arbitrary $n \in \mathbb{N}$ we define the set $\mathcal{A}_{k,n}$ by
    $$
    \mathcal{A}_{k,n} := \left\{A\in \mathcal{A} \colon \exists x \in \left[\frac{1}{n},1-\frac{1}{n}\right] \text{ such that } (D^+A(x) - D^-A(x)) \geq \frac{1}{k}\right\}.
    $$
    Then, obviously $\hat{\mathcal{A}} \subseteq \bigcup_{k \in \mathbb{N}}\bigcup_{n \in \mathbb{N}_{\geq 3}}\mathcal{A}_{k,n}$.
     We show that the sets $\mathcal{A}_{k,n}$ are closed with respect to $\Vert \cdot \Vert_\infty$ and 
     proceed as follows. Consider a sequence $(A_j)_{j \in \mathbb{N}}$ in $\mathcal{A}_{k,n}$ converging 
     uniformly to some $A \in \mathcal{A}$. 
     For each $j \in \mathbb{N}$ there exists some $x_j \in [\frac{1}{n},1-\frac{1}{n}]$ such that $D^{+}A_j(x_j) - D^{-}A_j(x_j) \geq \frac{1}{k}$. Since $[\frac{1}{n},1-\frac{1}{n}]$ is compact, there exists a subsequence 
     $(x_{j_\ell})_{\ell \in \mathbb{N}}$ and some $x^*$ such that $x_{j_\ell} \overset{\ell \rightarrow \infty}{\longrightarrow} x^* \in [\frac{1}{n},1-\frac{1}{n}]$. We can fix an arbitrary small $\Delta > 0$ 
     fulfilling that $x^* + \Delta$ is a point of continuity of $D^+A$ and $x^* - \Delta$ is a point of continuity of $D^-A$. Applying \cite[Theorem 5.1]{bernoulli} we therefore obtain that
    $$
    \lim_{\ell \rightarrow \infty}D^{+}A_{j_\ell}(x^*+\Delta) = D^{+}A(x^*+\Delta), \quad \lim_{\ell \rightarrow \infty}D^{-}A_{j_\ell}(x^*-\Delta) = D^{-}A(x^*-\Delta). 
    $$
    Since $(x_{j_\ell})$ converges to $x^*$ there exists some $\ell_0 \in \mathbb{N}$ such that for all 
    $\ell \geq \ell_0$ we have $|x_{j_\ell} - x^*| \leq \Delta$. Using convexity of $A_{j_\ell}$ yields 
    $$
    D^{+}A_{j_\ell}(x^*+\Delta) - D^{-}A_{j_\ell}(x^*-\Delta) \geq \frac{1}{k},
    $$
    hence, considering the limit $\ell \rightarrow \infty$ yields 
    $$
    D^{+}A(x^*+\Delta) - D^{-}A(x^*-\Delta) \geq \frac{1}{k}.
    $$
Since contiunuity points of $D^{+}A$ and $D^{-}A$ are dense in $\mathbb{I}$ we may choose $\Delta > 0$ arbitrarily small and conclude that
$$
D^{+}A(x^*) - D^{-}A(x^*) \geq \frac{1}{k}.
$$
Since smooth Pickands dependence functions are dense in $\mathcal{A}$ (see \cite{evc-mass}), 
$\mathcal{A}_{k,n}$ is nowhere dense in $\mathcal{A}$, and the result follows. 
\end{proof}
Translating to EVCs we obtain the main result of this section saying that topologically typical EVC 
have degenerated discrete component:
\begin{theorem}\label{cor:typical_ev}
    The set $\{C \in \mathcal{C}_{ev}\colon\mu_C^{dis}(\mathbb{I}^2) = 0\}$ is co-meager in $\mathcal{C}_{ev}$.
\end{theorem}
\begin{proof}
Immediate consequence of Lemma \ref{lem:extreme_homeom}, Lemma \ref{lem:pickands_co_meager} and Theorem \ref{thm:evc_regularity_meas_cop}.
\end{proof}
\begin{Rem}
Viewing Theorem \ref{cor:typical_ev} in context of Theorem \ref{thm:equiv_discrete_path_ev} implies that 
a typical EVC $C \in \mathcal{C}_{ev}$ does not exhibit pathological behavior, i.e.,  
$C \not\in \mathcal{C}_{ev,p}$.
\end{Rem}
Combining Corollary \ref{cor:typical_ev}, Corollary \ref{cor:abs_ev_meager} and Corollary \ref{cor:evc_full_sup} 
yields the following result on typical EVCs.
\begin{Cor}\label{cor:typical_evc}
A topologically typical Extreme Value copula $C$ has degenerated discrete component, is not absolutely continuous and 
has full support. In particular, $\partial_1C(x,y)$ exists in full $(0,1)^2$.
\end{Cor}
After having established the main results, we round off this section with some simple observations on 
exchangeable EVCs. 
\begin{Lemma}
The following assertions hold.
\begin{itemize}
    \item[(1)] $\mathcal{C}_{ev}$ is nowhere dense in $(\mathcal{C},d_\infty)$.
    \item[(2)] Exchangeable EVCs are nowhere dense in $(\mathcal{C}_e,d_\infty)$.
\end{itemize}
\end{Lemma}
\begin{proof}
Using the fact that $\mathcal{C}_{ev}$ is closed in $(\mathcal{C},d_\infty)$, 
for $C \in \mathcal{C}_{ev}$ setting $C_\varepsilon := (1 - \varepsilon) C + \varepsilon W$ we 
have that $C_\varepsilon \notin\mathcal{C}_{ev}$, implying that the interior of $\mathcal{C}_{ev}$ is empty.  
 The second assertion follows similarly using the fact that EVCs are positive quadrant dependent and $W$ is symmetric.
\end{proof}
\begin{Lemma}
The family $\mathcal{A}_s$ of symmetric Pickands dependence functions is nowhere dense in 
$(\mathcal{A},\Vert\cdot\Vert_\infty)$.
\end{Lemma}
\begin{proof}
    The family of symmetric Pickands dependence functions is obviously closed w.r.t. $\Vert \cdot \Vert_\infty$.
    Assume that $\mathcal{A}_s$ would not be nowhere dense. Then there exists some $\varepsilon > 0$ such that $B_\varepsilon(A) \subseteq \mathcal{A}_s$, whereby $B_\varepsilon(A)$ denotes the ball with center $A$ and radius $\varepsilon$ w.r.t. $\Vert \cdot \Vert_\infty$. Fixing an arbitrary asymmetric Pickands dependence function $D \in \mathcal{A}$ and setting $E := (1-\frac{\varepsilon}{3})A + \frac{\varepsilon}{3}D$ it follows that
     $E$ is an asymmetric Pickands dependence function with $E \in B_\varepsilon(A)$ and therefore $E \in \mathcal{A}_s$, 
     a contradiction.
 \end{proof}
\begin{Cor}
    The family of exchangeable EVC is nowhere dense in $(\mathcal{C}_{ev},d_\infty)$.
\end{Cor}
We conclude this section with an example of an atypical EVC:
\begin{Ex}\label{ex:discret_pickands}
Defining the probability measure
$$
\vartheta := \frac{1}{5}\,\delta_\frac{1}{4} + \frac{3}{5}\,\delta_\frac{1}{2} + \frac{1}{5}\,\delta_\frac{3}{4},
$$
clearly $\vartheta \in \mathcal{P}_\mathcal{A}$, and the corresponding Pickands dependence function is given by
$$
A(t) := \begin{cases}
    1-t,& \text{ if } t\in [0,\frac{1}{4})\\
    -\frac{3}{5}t+\frac{9}{10},& \text{ if } t\in [\frac{1}{4},\frac{1}{2})\\
     \frac{3}{5}t+\frac{3}{10},& \text{ if } t\in [\frac{1}{2},\frac{3}{4})\\
     t,& \text{ if } t\in [\frac{3}{4},1].
\end{cases}
$$
According to Corollary \ref{cor:typical_evc} the EVC $C_A$ is atypical, since it has a non-degenerated 
discrete component which is concentrated on the graphs of the functions $f^{\frac{1}{4}}$, $f^{\frac{1}{2}}$ and $f^{\frac{3}{4}}$. Moreover, $C_A$ does not have full support, since $\mathrm{supp}(\mu_{C_A}) = \{(x,y) \in \mathbb{I}^2 \colon f^{\frac{1}{4}}(x) \leq y \leq f^{\frac{3}{4}}(x)\}$. A sample of the copula $C_A$ is depicted in Figure \ref{fig:discrete_pickands_sample}.
\end{Ex}
\begin{figure}[!ht]
	\centering
	\includegraphics[width=1\textwidth]{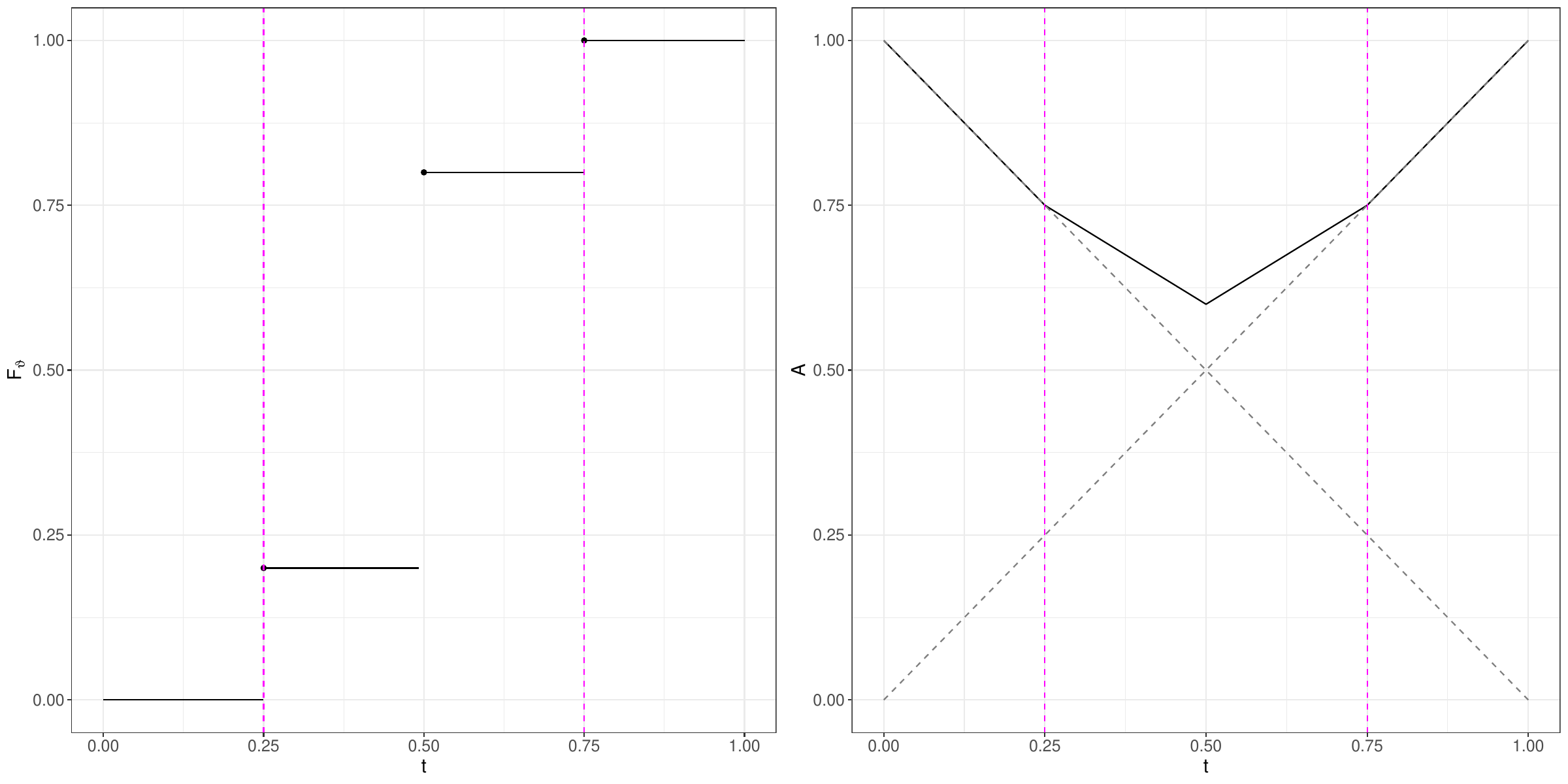}
	\caption{Graphs of the distribution function $F_\vartheta$ of the Pickands dependence measure $\vartheta$ (left) according to Example \ref{ex:discret_pickands} and Pickands dependence function $A$ (right) associated with it. The dashed magenta lines mark $L_\vartheta= \frac{1}{4}$ and $R_\vartheta = \frac{3}{4}$ according to equations \eqref{eq:L and R measure} and \eqref{eq:L and R Pickands}, respectively.}
\label{fig:discrete_pickands}
\end{figure}
\begin{figure}[!ht]
	\centering
	\includegraphics[width=1\textwidth]{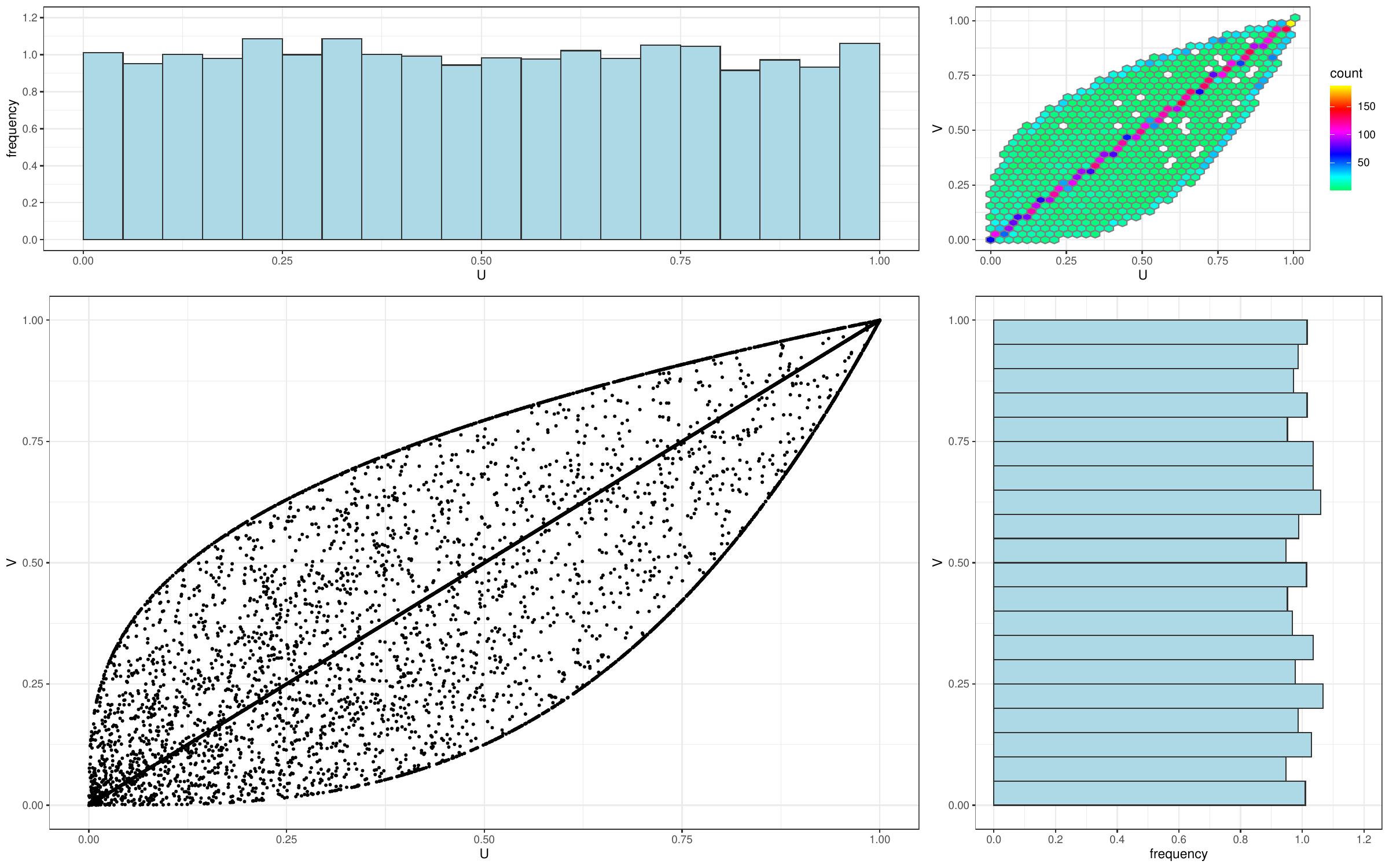}
	\caption{Sample of size 10000 of the EVC $C_A$, where $A$ is the Pickands dependence function considered in Example \ref{ex:discret_pickands}, its histogram and the two marginal histograms. The sample has been generated via conditional inverse sampling.}
\label{fig:discrete_pickands_sample}
\end{figure}
\clearpage
\section*{Acknowledgement}
\noindent The first author gratefully acknowledges the support of Red Bull GmbH
within the ongoing Data Science collaboration with the university of Salzburg. The second author gratefully acknowledges the support of the WISS 2025 project ‘IDA-lab Salzburg’
(20204-WISS/225/197-2019 and 20102-F1901166-KZP).

\appendix
\section{Auxiliary results}
\noindent The proof of the following lemma is included for the sake of completeness:
\begin{Lemma}\label{lem:homeom_preserve_cat}
    Let $E_1$ and $E_2$ be two topological spaces and $f \colon E_1 \rightarrow E_2$ be a homeomorphism. Then the following assertions hold:
    \begin{itemize}
        \item[(1)] $G \subseteq E_1$ is nowhere dense in $E_1$, if and only if $f(G)$ is nowhere dense $E_2$.
        \item[(2)] If $G \subseteq E_1$ is of first Baire category in $E_1$, then $f(G)$ is of first Baire category in $E_2$.
        \item[(3)] If $G \subseteq E_1$ is of second Baire category in $E_1$, then $f(G)$ is of second Baire 
        category in $E_2$.
        \item[(4)] If $G \subseteq E_1$ is co-meager in $E_1$, then $f(G)$ is co-meager in $E_2$
    \end{itemize}
\end{Lemma}
\begin{proof}
    $(1)$: Let $G \subseteq E_1$ be nowhere dense. Then $\mathrm{int}(\overline{G}) = \emptyset$. 
    Since homeomorphisms preserve interior and closure, we have that $\mathrm{int}(\overline{f(G)}) = 
    \mathrm{int}(f(\overline{G})) = f(\mathrm{int}(\overline{G})) = f(\emptyset) = \emptyset$.
    Now assume that $f(G)$ is of first Baire category. Then $\emptyset = f(\mathrm{int}(\overline{G}))$ and using the fact that $f$ is injective we have that $\emptyset = \mathrm{int}(\overline{G})$, implying that $G$ is nowhere dense in $E_1$.\\
    $(2)$: Let $G \subseteq E_1$ be of first Baire category. Then there exist countably many nowhere dense sets $(G_n)_{n \in \mathbb{N}}$ in $E_1$ such that $G = \bigcup_{n \in \mathbb{N}}G_n$. Since $f$ is surjective we have that $f(G) = f\left(\bigcup_{n \in \mathbb{N}}G_n\right) = \bigcup_{n \in \mathbb{N}}f(G_n)$. Since $G_n$ is nowhere dense in $E_1$, applying $(1)$ yields that $f(G_n)$ is nowhere dense in $E_2$. Thus, $f(G)$ is the union of nowhere dense sets and therefore of first Baire category.\\
    $(3)$: Let $G \subseteq E_1$ be a set of second category. Suppose that $f(G)$ is not of second category, i.e., 
    that it is of first Baire category. Thus, $f(G) = \bigcup_{n\in\mathbb{N}}A_n$ for nowhere dense sets $A_1,A_2,...$ in $E_2$. Since $f$ is a homeomorphism, we can find a set $G_n$ such that $A_n = f(G_n)$. Applying $(1)$ yields that $G_n$ is nowhere dense and therefore $f(G) = \bigcup_{n\in\mathbb{N}}A_n = \bigcup_{n\in\mathbb{N}}f(G_n) = f\left(\bigcup_{n\in\mathbb{N}}G_n\right)$. Again, using that $f$ is a homeomorphism yields that
    $G = \bigcup_{n\in\mathbb{N}}G_n$ and therefore $G$ would be of first Baire category in $E_1$. A contradiction.\\
    $(4)$: Let $G \subseteq E_1$ be a co-meager set in $E_1$. Then its complement $G^c$ is of first Baire category. 
    Since $f$ is a homeomorphism we have that $f(G)^c \subseteq f(G^c)$. Applying $(2)$, $f(G^c)$ is of first Baire category in $E_2$ and since subsets of sets of first Baire category are of first Baire category too, $f(G)^c$ is of first 
    Baire category, implying that $f(G)$ is co-meager.
\end{proof}
We prove the interrelation between the Pickands dependence function $A \in \mathcal{A}$ and Pickands dependence 
measure $\vartheta \in \mathcal{P}_\mathcal{A}$ according to equation \eqref{eq:gamma}. 
Note that this interrelation goes back to \cite{beirlant2004,haan1977,Pickands}, we here only include a quick proof
for the sake of completness.
\begin{Lemma}\label{lem:measure_pickands}
    Suppose that $A \in \mathcal{A}$ and that $\vartheta \in \mathcal{P}_\mathcal{A}$ is the corresponding Pickands 
    dependence measure. 
    Then \begin{equation}\label{eq:formula_meas_pickands}
    A(t) := 1-t + 2\int_{[0,t]}\vartheta([0,z])\mathrm{d}\lambda(z)
    \end{equation}
    holds for all $t \in \mathbb{I}$.
\end{Lemma}
\begin{proof}
Following \cite{Pickands}, the interrelation between the Pickands dependence function $A$ and $\vartheta$ translates to
\begin{equation}
A(t) = 2 \int_{[0,1]} \max\{ts,(1-t)(1-s)\} d\vartheta(s).
\end{equation}
The last equation boils down to equation (3) in \cite{evc-mass} as follows: 
\begin{eqnarray*}
\frac{A(t)}{2} &=& (1-t) \int_{[0,1-t]} (1-s) d\vartheta(s) \, + \, t \int_{(1-t,1]} s d\vartheta(s) \\
&=& (1-t) \, \vartheta([0,1-t]) - (1-t) \int_{[0,1-t]} s d\vartheta(s) \, + \, t \int_{(1-t,1]} s d\vartheta(s) \\
&=& (1-t) \, \vartheta([0,1-t]) - \int_{[0,1-t]} s d\vartheta(s) \, + \, t \underbrace{\int_{[0,1]} s d\vartheta(s)} _{=\tfrac{1}{2}} \\
&=& (1-t) \, \vartheta([0,1-t]) + \, \frac{t}{2} \,- \int_{[0,1-t]} s d\vartheta(s).
\end{eqnarray*} 
The latter integral, however, can easily be simplified to (it is the expectation of the probability measure
$\vartheta'(E):=\frac{\vartheta(E)}{\vartheta([0,1-t])}$ and hence can be expressed as integral over the 
corresponding distribution function)
\begin{eqnarray*}
\int_{[0,1-t]} s d\vartheta(s) &=& \vartheta([0,1-t]) \int_{[0,1-t]} (1-\vartheta'([0,s])) d\lambda(s) \\
&=& \vartheta([0,1-t])(1-t) - \int_{[0,1-t]} \vartheta([0,s])ds.
\end{eqnarray*}  
Altogether we therefore get
\begin{eqnarray*}
\frac{A(t)}{2} &=& \frac{t}{2} + \int_{[0,1-t]} \vartheta([0,s])d\lambda(s),
\end{eqnarray*}
implying
\begin{equation}
A(t)= t + 2 \int_{[0,1-t]} \vartheta([0,s])d\lambda(s),
\end{equation}
and the latter is equivalent to equation (3) in \cite{evc-mass}.
\end{proof}
The next lemma states that weak convergence of measures in $\mathcal{P}_\mathcal{A}$ is equivalent to uniform 
convergence of the corresponding EVC (in fact, even to weak conditional convergence, 
see \cite{bernoulli}).
\begin{Lemma}\label{lem:equiv_conv_meas_evc}
    Let $\vartheta,\vartheta_1,\vartheta_2,...\in \mathcal{P}_\mathcal{A}$ be Pickands dependence measures and $A,A_1,A_2,...\in\mathcal{A}$ as well as $C,C_1,C_2,...\in\mathcal{C}_{ev}$ the corresponding Pickands dependence functions and EVCs, respectively. Then the following three conditions are equivalent:
    \begin{itemize}
        \item[(i)]  $\vartheta_n \overset{n\rightarrow\infty}{\longrightarrow} \vartheta$ weakly on $\mathbb{I}$,
        \item[(ii)] $A_n \overset{n\rightarrow\infty}{\longrightarrow} A$ \text{ uniformly on } $\mathbb{I}$,
        \item[(iii)] $C_n \overset{n\rightarrow\infty}{\longrightarrow} C$ \text{ uniformly on } $\mathbb{I}^2$.
    \end{itemize}
\end{Lemma}
\begin{proof}
    Considering \cite[Theorem 5.1]{bernoulli} it suffices to show the equivalence of $(i)$ and $(ii)$. 
    The implication $(i) \Rightarrow (ii)$ is a direct consequence of Lemma \ref{lem:measure_pickands} and dominated convergence. On the other hand, again using \cite[Theorem 5.1]{bernoulli}, we have that 
    $D^+A_n(t) \overset{n\rightarrow \infty}{\longrightarrow}D^+A(t)$ for every continuity point 
    $t$ of $D^+A$, which, applying equation \eqref{eq:formula_meas_pickands} directly yields $(i)$.
\end{proof}
For establishing the regularity results for EVCs summarized in Theorem \ref{thm:evc_regularity_meas_cop} the following 
technical lemma will be used:
\begin{Lemma}\label{lem:help_regularity_measure}
    Let $F\colon \mathbb{I} \rightarrow \mathbb{I}$ be a strictly increasing, absolutely continuous distribution 
    function, $\mu_G$ be a finite measure on $\mathcal{B}(\mathbb{I})$ with measure-generating function $G$ 
    and fulfilling $\mu_G((0,1))>0$. Furthermore let $\mu_H$ be the measure induced by the measure-generating function $H\colon \mathbb{I} \rightarrow \mathbb{I}$ defined by $H := F \cdot G$. Then the following three assertions hold:
    \begin{itemize}
        \item[(i)] If $\mu_G$ is singular, then $\mu_H$ has non-degenerated singular component.
        \item[(ii)] If $\mu_G$ is discrete in $(0,1)$, then $\mu_H^{sing}(\mathbb{I}) = 0$.
        \item[(iii)] If $\mu_G$ has a point mass in $(0,1)$, then so does $\mu_H$.
    \end{itemize}
\end{Lemma}
\begin{proof}
  (i) Continuity of $H$ implies that $H$ is continuous on $\mathbb{I}$. 
  Letting $f$ denote the density of $F$ and considering the derivative of $H$ yields
  $$
  H' = F'G + FG' = fG + FG'
  $$
  $\lambda$-almost everywhere on $\mathbb{I}$. Using singularity of $G$ therefore implies $H' = fG$ $\lambda$-almost everywhere on $\mathbb{I}$.\\
  If $H$ had no singular component, it would be absolutely continuous with density $f \cdot G$.
  Therefore, considering an interval $(x_1,x_2) \subseteq (0,1]$ with $G(x_1) < G (x_2)$ and using the fact 
  that $F$ is strictly increasing and greater than $0$ on $(0,1]$, it follows that
    \begin{align*}
        H(x_2) - H(x_1)  &= \int_{[x_1,x_2]}f(s)G(s) \mathrm{d}\lambda(s) \\&\leq G(x_2)[F(x_1) - F(x_2)] \\& <
        G(x_2) F(x_2) - G(x_1)F(x_1) \\&=
        H(x_2) - H(x_1),
    \end{align*}
   a contradiction, so $H$ has non-degenerated singular component.\\
(ii) Assume that $\mu_G$ is discrete. Then there exists a (finite or countably infinite) index-set $I$ and
 numbers $q_i\in [0,1)$ and $a_i \in \mathbb{I}$ for every $i \in I$, such 
 that $\mu_G = \sum_{i \in I}a_i \delta_{q_i}$ holds. Then obviously $H$ is given by
$$
H(t)=F(t)G(t) = \sum_{i \in I}a_i F(t) \mathbf{1}_{[0,t]}(q_i)
$$
for every $t \in (0,1)$. 
$H$ is continuous outside the set $\{q_i: i \in I\}$, so the discrete component $H^{dis}$ of $H$ is given by $H^{dis}(t) =  \sum_{i \in I}a_i F(q_i) \delta_{q_i}([0,t])$. Defining $\psi(t) := H(t)- H^{dis}(t)$ and letting $f$ again 
denote the density of $F$, working with the definition of $\psi$ and applying Fubini's theorem for non-negative measurable functions, we obtain that
\begin{align*}
\psi(t) &= \sum_{i \in I}a_i[F(t)-F(q_i)] \delta_{q_i}([0,t]) \\ & = 
\sum_{i \in I}a_i\int_{(q_i,t]}f(s) \mathrm{d}\lambda(s) \delta_{q_i}([0,t]) \\ & = 
\int_{\mathbb{I}}\sum_{i \in I}a_i f(s)\mathbf{1}_{(q_i,t]}(s) \delta_{q_i}([0,t]) \mathrm{d}\lambda(s) \\ & =
\int_{[0,t]}f(s)\sum_{q_i < s}a_i \mathrm{d}\lambda(s).
\end{align*}
Observing that, defining $g(s) := f(s)\sum_{q_i < s}a_i \geq 0$ for every $s \in \mathbb{I}$, we obviously have
that $g\in L^1(\mathbb{I},\mathcal{B}(\mathbb{I}),\lambda)$ it follows that $\psi$ is absolutely continuous
on $\mathbb{I}$. In other words: the singular component of $H$ is degenerated.\\
(iii) The third assertion directly follows form the fact that in the case of $x_0$ being a point mass of $\mu_G$ we have 
$$
\mu_H(\{x_0\}) = F(x_0)\cdot\mu_G(\{x_0\}) >0.
$$
\end{proof}
\begin{Lemma}\label{lem:conv_semi-pickands}
    Let $\nu$ be a measure on $\mathcal{B}(\mathbb{I})$ with $\nu(\mathbb{I}) \leq 1$. Then the function $g_\nu \colon \mathbb{I} \rightarrow [0,\infty)$ defined by
    $$
    g_\nu(t) := 1-t + 2\int_{[0,t]}\nu([0,z]) \mathrm{d}\lambda(z)
    $$
    for all $t \in \mathbb{I}$ is convex and $1$-Lipschitz continuous.
\end{Lemma}
\begin{proof}
The function $g_\nu$ is absolutely continuous by definition and (one version of) its density $\mathcal{k}_\nu \colon \mathbb{I} \rightarrow [-1,1]$ is given by $\mathcal{k}_\nu(t) = -1 + 2\nu([0,t]) \in [-1,1]$, implying that
$$
|g_\nu(t_1) - g_\nu(t_2)| \leq \int_{[x_1,x_2]} |\mathcal{k}_\nu(s)| \mathrm{d}\lambda(s) \leq |t_1-t_2|
$$
for every $t_1,t_2 \in \mathbb{I}$, whereby $x_1 = \min\{t_1,t_2\}$ and $x_2 = \max\{t_1,t_2\}$. In other words: 
$g_\nu$ is $1$-Lipschitz continuous.
Working with the fact that $z \mapsto \nu([0,z])$ is non-decreasing, the function $t \mapsto \int_{[0,t]}\nu([0,z])\mathrm{d}\lambda(z)$ is convex and hence, $g_\nu$ is as sum of two convex functions convex as well.
\end{proof}
The following lemma generalizes Lemma 5 in \cite{evc-mass}.
\begin{Lemma}\label{lem:conv_G}
    Let $f\colon \mathbb{I} \rightarrow [0,\infty)$ be a convex function fulfilling $D^+f(0) \geq -1$ and $f(0) = 1$. Then the function $G_f \colon [0,1) \rightarrow \mathbb{R}$, defined by
    $$
    G_f(t) := D^+f(t)(1-t) + f(t)
    $$
    is non-decreasing, non-negative and right-continuous.
\end{Lemma}
\begin{proof}
     Convexity of $f$ yields that $f$ is continuous and that $D^+f$ is right-continuous 
     (see \cite{Kannan1996,Pollard2001}), hence right-continuity of $G_f$ follows.
     For $0 \leq t_1 < t_2 <1$ convexity of $f$ implies $D^+f(t_1) \leq D^+f(t_2)$. 
     Therefore, setting $\delta := D^+f(t_2)(1-t_2) + f(t_2) - (D^+f(t_1)(1-t_1) + f(t_1))$, we obtain that
     \begin{align*}
     \delta &\geq D^+f(t_2)(1-t_2) + f(t_1) + D^+f(t_1)(t_2-t_1) - (D^+f(t_1)(1-t_1) + f(t_1)) \\&\geq
      D^+f(t_1)(1-t_2) + f(t_1) + D^+f(t_1)(t_2-t_1) - (D^+f(t_1)(1-t_1) + f(t_1)) = 0,
    \end{align*}
    implying that $G_f$ is non-decreasing. Finally, non-negativity follows via $G_f(t) \geq G_f(0) = D^+f(0)(1-0) + f(0) = 1 + D^+f(0) \geq 0$.
\end{proof}

\noindent \textbf{Proof of Theorem 4.6:}
\begin{proof}
    (i): If $\vartheta$ is absolutely continuous on $(0,1)$, applying equation \eqref{eq:der_pick_meas} yields that 
    $D^+A$ is absolutely continuous too. Since $A$ is Lipschitz-continuous and finite sums of absolutely continuous functions are absolutely continuous, $G_A$ is absolutely continuous on $[0,1)$, i.e., there exists some function 
    $g \in L^1(\mathbb{I},\mathcal{B}(\mathbb{I}),\lambda)$, such that
    $$
    G_A(t) = G_A(0) + \int_{[0,t]} g(s) \mathrm{d}\lambda(s),
    $$
    holds for every $t \in [0,1)$. Using the fact that for fixed $x \in (0,1)$ the mapping
     $\varphi_x: y \mapsto \frac{\log(x)}{\log(xy)}$ is a diffeomorphism of $(0,1)$, applying change of coordinates yields that
    $$
    G_A\left(\varphi_x(y)\right) = G_A(0) + \int_{[0,y]} g(\varphi_x(s))\left|\frac{\log(x)}{s\log^2(xs)}\right| \mathrm{d}\lambda(s),
    $$
  implying that $y \mapsto G_A\left(\varphi_x(y)\right)$ is absolutely continuous on $[0,1)$. 
  Since $C_A$ is Lipschitz-continuous and products of absolutely continuous functions on compact 
  intervals   are absolutely continuous again (see \cite[Section 7.5, Exercise 7]{gariepy})
  the function $y \mapsto K_A(x,[0,y])$ according to equation \eqref{eq:ev_markov_kernel} is absolutely continuous 
  on $(0,1)$. Therefore, applying disintegration shows that $C_A$ is absolutely continuous.\\
  Considering the reverse implication, we know that there exists some Borel set $\Lambda \subseteq (0,1)$ with 
  $\lambda(\Lambda) = 1$ such that for every $x \in \Lambda$ the measure $K_A(x,\cdot)$ is absolutely continuous. 
  Fixing $x \in \Lambda$, using that $y \mapsto G_A\left(\varphi_x(y)\right)$ induces a probability measure and $y \mapsto \frac{C(x,y)}{x}$ is a strictly increasing absolutely continuous distribution function, 
  $y \mapsto G_A\left(\varphi_x(y) \right)$ is absolutely continuous. In fact, if this was not the case, then 
  $y \mapsto G_A\left(\varphi_x(y)\right)$ would have a discrete or singular component and thus, using 
  the fact that $K_A(x,\cdot)$ is absolutely continuous together with Lemma \ref{lem:help_regularity_measure} would yield a contradiction. 
   Proceeding as in the previous case, applying change of coordinates shows 
   absolute continuity of $G_A$ and therefore, considering that the function $A$ is Lipschitz-continuous and 
   $t \mapsto 1-t$ is smooth, absolute continuity of $D^+A$ follows. Applying equation \eqref{eq:der_pick_meas} now 
   completes the proof of the first assertion.\\
(ii) The second assertion has already been proved in Lemma \ref{lem:mass_graph}.\\
(iii) For proving assertion number three assume that $\vartheta \in \mathcal{P}_\mathcal{A}$ has a singular component, 
i.e., $\vartheta = \vartheta^{sing} + \mu$ with $\vartheta^{sing}(\mathbb{I}) > 0$ and $\mu$ being the non-singular 
component. Working with equation \eqref{eq:ev_markov_kernel} for fixed $x\in (0,1)$ yields 
$$
K_C(x,[0,y])= \frac{C(x,y)}{x}\left[G_A^{sing}\left(\frac{\log(x)}{\log(xy)}\right) + G_A^\mu\left(\frac{\log(x)}{\log(xy)}\right)\right],
$$
where $y \in (0,1)$, $G_A^{i}(t) =D^+A_i(t)(1-t) + A_i\left(t\right)$, $i \in \{sing,\mu\}$, and $A_i$ 
according to equation \eqref{eq:formula_meas_pickands} with the measures $\vartheta_{sing}$ and $\mu$ being 
used instead of $\vartheta$. Applying Lemma \ref{lem:conv_semi-pickands} each $A_i$ is convex and thus, using Lemma \ref{lem:conv_G}, $G_A^i$ are non-negative, non-decreasing and right-continuous. Furthermore applying that 
$y \mapsto \varphi_x(y)$ is non-decreasing and right-continuous as well, $G_A^i$ induces a measure for $i \in \{sing,\mu\}$. Moreover, using that $\varphi_x$ is continuous, strictly increasing and that its derivative is bounded from 
below by $\ell = -\frac{e^2}{4}x\log(x) > 0$ (see \cite[Example 4]{evc-mass}), yields that 
$\frac{\partial}{\partial y}G_A^{sing}\left( \varphi_x(y)\right) = 0$ for $\lambda$-almost every $y \in (0,1)$, implying
that $y \mapsto G_A^{sing}\left(\varphi_x(y)\right)$ is continuous and singular.
 Applying Lemma \ref{lem:help_regularity_measure} yields that 
 $y \mapsto \frac{C(x,y)}{x}\,G_A^{sing}\left( \varphi_x(y)\right)$ has a non-degenerated singular component and therefore $K_C^{sing}(x,\mathbb{I}) > 0$. Applying disintegration yields the desired result. \\
To prove the reverse implication assume that $\vartheta^{sing}(\mathbb{I})=0$, i.e., $\vartheta = \vartheta^{dis} + \vartheta^{abs}$. Let $x \in (0,1)$ be arbitrary but fixed. Again working with equation 
\eqref{eq:ev_markov_kernel} yields that
$$
K_C(x,[0,y]) = \frac{C(x,y)}{x}\left[G_A^{abs}\left(\frac{\log(x)}{\log(xy)}\right) + G_A^{dis}\left(\frac{\log(x)}{\log(xy)}\right)\right],
$$
where $G_A^{abs}$ and $G_A^{dis}$ are as in the previous case only for the absolutely continuous and 
discrete components of $\vartheta$, respectively. 
Using the same arguments as in the proof of the first assertion shows that 
$y \mapsto \frac{C(x,y)}{x}G_A^{abs}\left(\varphi_x(y)\right)$ is absolutely continuous on $(0,1)$. 
Considering $G_A^{dis}$, notice that according to Lemma \ref{lem:conv_semi-pickands} $A^{dis}$ is Lipschitz-continuous 
and, since $D^+A^{dis}$ is discrete and $t \mapsto (1-t)$ is smooth, $G_A^{dis}$ has no singular component. 
If $y \mapsto G_A^{dis}\left(\frac{\log(x)}{\log(xy)}\right) =: Q(y)$ had a non-degenerated singular component, i.e., $\mu_Q = \mu_Q^r + \mu_Q^s$, whereby $\mu_Q$ denotes the measure induced by $Q$ and $\mu_Q^s$ the non-degenerated singular 
component of $\mu_Q$ and $\mu_Q^r$ summarizes the absolutely continuous and discrete component. Denote
 the measure generating functions of $\mu_Q^s$ and $\mu_Q^r$ by $Q^s$ and $Q^r$, respectively. Then
$$
G_A^{dis}\left(t\right) = G_A^{dis}\left(\frac{\log(x)}{\log(xx^{\frac{1}{t}-1})}\right) = Q(x^{\frac{1}{t}-1}) = Q^r(x^{\frac{1}{t}-1}) + Q^s(x^{\frac{1}{t}-1})
$$
and the function $k_x(t) := x^{\frac{1}{t}-1}$ is differentiable for every $t \in (0,1)$. 
Applying \cite[Lemma 7.25]{rudin1974} yields that $k_x$ maps sets of full $\lambda$-measure to sets of full $\lambda$-measure.
Applying the obvious fact that $t \mapsto Q^s(x^{\frac{1}{t}-1})$ is continuous and that
$$\frac{\mathrm{d}}{\mathrm{d}t}Q^s(x^{\frac{1}{t}-1}) = (Q^s)'(x^{\frac{1}{t}-1})\frac{x^{\frac{1}{t}-1}}{t^2\log(x)} = 0$$
for $\lambda$-almost every $t \in \mathbb{I}$, the measure induced by $G_A^{dis}$ would have a non-degenerated singular component; a contradiction. 
Therefore $y \mapsto G_A^{dis}\left(\varphi_x(y)\right)$ has degenerated singular component and thus, multiplying 
$y \mapsto G_A^{dis}\left(\varphi_x(y) \right)$ with the strictly increasing, absolutely continuous distribution 
function $y \mapsto \frac{C(x,y)}{x}$ and working with the fact that 
$y \mapsto G_A^{dis}\left(\varphi_x(y)\right)$ only has non-degenerated absolutely continuous and 
discrete component, proceeding analogously as in the proof of the first part of the Lemma and applying Lemma \ref{lem:help_regularity_measure} $(ii)$ completes the proof of the last assertion.
\end{proof}

The next lemma shows that we can normalize sequences of weakly converging probability measures with limit 
in $\mathcal{P}_\mathcal{A}$ in such a way that all elements of the sequence are in 
$\mathcal{P}_\mathcal{A}$ while preserving weak convergence to the limit.
\begin{Lemma}\label{lem:normalizing_pickands}
Let $\mu_1,\mu_2,... \in\mathcal{P}(\mathbb{I})$ and $\vartheta \in \mathcal{P}_\mathcal{A}$ fulfill that the sequence
$(\mu_n)_{n \in \mathbb{N}}$ converges weakly to $\vartheta$. 
Then there exist constants $\alpha_1,\alpha_2,... \in (0,1)$ and $\beta_1,\beta_2,... \in (0,1)$ such that the measures
\begin{equation}\label{eq:norm_theta}
   \vartheta_n := \begin{cases}
    (1-\alpha_n)\delta_0+ \alpha_n\mu_n, &\text{ if } \mathbb{E}(\mu_n) > \frac{1}{2}\\
    \mu_n,  &\text{ if } \mathbb{E}(\mu_n) = \frac{1}{2}\\
    (1-\beta_n)\delta_1 + \beta_n\mu_n, &\text{ if } \mathbb{E}(\mu_n) < \frac{1}{2}
\end{cases} 
\end{equation}
fulfill the following assertions:
\begin{itemize}
    \item[(i)] $\vartheta_n \in \mathcal{P}_\mathcal{A}$ for every $n \in \mathbb{N}$.
    \item[(ii)] The sequence $(\vartheta_n)_{n \in \mathbb{N}}$ converges weakly to $\vartheta$.
\end{itemize}
\end{Lemma}
\begin{proof}
For a given $\mu_n$ set $\alpha_n := \frac{1}{2\mathbb{E}(\mu_n)}$ if $\mathbb{E}(\mu_n) > \frac{1}{2}$, 
$\beta_n := \frac{1}{2(1-\mathbb{E}(\mu_n))}$ if $\mathbb{E}(\mu_n) < \frac{1}{2}$ and define $\vartheta_n$ 
according to equation \eqref{eq:norm_theta}.
Then for $\mathbb{E}(\mu_n) > \frac{1}{2}$ we have that
$$
\mathbb{E}(\vartheta_n) = (1-\alpha_n) \mathbb{E}(\zeta_1) + \alpha_n\mathbb{E}(\mu_n) = \frac{\mathbb{E}(\mu_n)}{2\mathbb{E}(\mu_n)} = \frac{1}{2}
$$
and in the case $\mathbb{E}(\mu_n) < \frac{1}{2}$, we obtain that
$$
\int_\mathbb{I} s\mathrm{d}\vartheta_n(s) = (1-\beta_n) + \beta_n\mathbb{E}(\mu_n) =
\left(1-\frac{1}{2(1-\mathbb{E}(\mu_n))}\right) + \frac{\mathbb{E}(\mu_n)}{2(1-\mathbb{E}(\mu_n))} = 
\frac{1}{2}.
$$
This already proves the first assertion.\\
To show the second one we proceed as follows: Defining 
$$
I^>:=\{n \in \mathbb{N}: \mathbb{E}(\mu_n)>1/2\},
I^<:=\{n \in \mathbb{N}: \mathbb{E}(\mu_n)<1/2\},
I^=:=\{n \in \mathbb{N}: \mathbb{E}(\mu_n)=1/2\}
$$
we obviously have $I^> \cup I^< \cup I^==\mathbb{N}$.
If $I^>$ contains infinitely many elements, increasingly enumerating its elements by $i^>_1,i^>_2,\ldots$, it follows immediately that $(\alpha_{i^>_j})_{j \in \mathbb{N}}$ converges to $1$. Furthermore weak convergence implies that
\begin{align*}
\lim_{j \rightarrow \infty}\int_\mathbb{I}h(x) \mathrm{d}\vartheta_{i^>_j}(x) &= \lim_{j \rightarrow \infty}\left[(1-\alpha_{i^>_j})\int_\mathbb{I}h(x) \mathrm{d}\delta_0(x) + \alpha_{i^>_j}\int_\mathbb{I}h(x) \mathrm{d}\mu_{i^>_j}(x)\right] \\&= \lim_{j \rightarrow \infty}\left[(1-\alpha_{i^>_j})h(0) + \alpha_{i^>_j}\int_\mathbb{I}h(x) \mathrm{d}\mu_{i^>_j}(x)\right] \\&=  \int_\mathbb{I}h(x) \mathrm{d}\vartheta(x)
\end{align*}
holds for every continuous function $h \colon \mathbb{I} \rightarrow \mathbb{R}$.
For the case that $I^<$ is infinite we proceed analogously to show that $(\beta_{i^<_j})_{j \in \mathbb{N}}$ converges to $1$ and that $(\vartheta_{i^<_j})_{j \in \mathbb{N}}$ converges weakly to $\vartheta$.
Since for $n \in I^=$ we have $\vartheta_n=\mu_n$ it altogether follows that the sequence $(\vartheta_n)_{n \in \mathbb{N}}$ converges weakly to $\vartheta$, and the proof is complete.
\end{proof}
The next lemma constitutes that both, the family of Pickands dependence measures whose discrete component has full support as well as the family of Pickands dependence measures whose singular component has full support are dense in $\mathcal{P}_\mathcal{A}$.
\begin{Lemma} \label{lem:ev-measure-approximaton}
The following assertions hold:
\begin{itemize}
    \item[(i)] The family of $\vartheta \in \mathcal{P}_\mathcal{A}$ with $\mathrm{supp}(\vartheta^{dis}) = \mathbb{I}$ is dense in $(\mathcal{P}_\mathcal{A},\tau_w)$.
    \item[(ii)] The family of $\vartheta \in \mathcal{P}_\mathcal{A}$ with $\mathrm{supp}(\vartheta^{sing}) = \mathbb{I}$ is dense in $(\mathcal{P}_\mathcal{A},\tau_w)$.
\end{itemize}
\end{Lemma}
\begin{proof}
    To prove the first assertion we proceed as follows: Let $\vartheta \in \mathcal{P}_\mathcal{A}$ be arbitrary but fixed and choose some discrete measure $m_1 \in \mathcal{P}(\mathbb{I})$ with full support. Then defining $\mu_n := (1-\frac{1}{n})\vartheta + \frac{1}{n}m_1$ yields that $\mu_n \in \mathcal{P}(\mathbb{I})$, $\mathrm{supp}(\mu_n^{dis}) = \mathbb{I}$ and the sequence $(\mu_n)_{n \in \mathbb{N}}$ converges to $\vartheta$ weakly. Normalizing the sequence 
    $(\mu_n)_{n \in \mathbb{N}}$ in the sense of Lemma \ref{lem:normalizing_pickands}, we obtain measures $\vartheta_n \in \mathcal{P}_\mathcal{A}$ with $\mathrm{supp}(\vartheta_n^{dis}) = \mathbb{I}$ such that
    $(\vartheta_n)_{n \in \mathbb{N}}$ converges weakly to $\vartheta$.\\
    On the other hand, choosing a singular probability measure $m_2 \in \mathcal{P}(\mathbb{I})$ with $\mathrm{supp}(m_2) = \mathbb{I}$ (see \cite{Sanchez2012,stromberg1965} for an example for $m_2$) and setting $\nu_n := (1-\frac{1}{n})\vartheta + \frac{1}{n}m_2$ yields the desired result.
\end{proof}
\end{document}